\documentclass[10pt,a4paper,english]{article}

\usepackage{amsfonts,amssymb,stmaryrd,amsmath,amsthm,amssymb}
\usepackage{graphicx,color}
\usepackage{multimedia}
\usepackage{float}
\usepackage{placeins}

\usepackage{algorithm}

\usepackage[a4paper]{geometry}
\newtheorem{theorem}{Theorem}
\newtheorem{lemma}{Lemma}
\newtheorem{remark}{Remark}
\newtheorem{proposition}{Proposition}
\newtheorem{corollary}{Corollary}
\newtheorem{definition}{Definition}

\newcommand \p {\partial}

\newcommand \R {\mathbb{R}}
\newcommand \N {\mathbb{N}}
\renewcommand \L {\mathrm{L}}
\newcommand \W {\mathrm{W}}
\newcommand \WW {\mathbf{W}}
\newcommand \WWW {\mathbb{W}}
\newcommand \V {\mathrm{V}}
\newcommand \VV {\mathbf{V}}
\newcommand \LL {\mathbf{L}}
\newcommand \LLL {\mathbb{L}}

\renewcommand \H {\mathrm{H}}

\newcommand \I {\mathrm{I}}

\renewcommand \d {\mathrm{d}}

\renewcommand \det {\mathrm{det}}
\newcommand \trace {\mathrm{tr}}

\newcommand \cof {\mathrm{cof}}

\DeclareMathOperator{\divg}{div}

\begingroup\makeatletter\ifx\SetFigFont\undefined%
\gdef\SetFigFont#1#2#3#4#5{%
  \reset@font\fontsize{#1}{#2pt}%
  \fontfamily{#3}\fontseries{#4}\fontshape{#5}%
  \selectfont}%
\fi\endgroup%

\title{Almost global existence of weak solutions for the nonlinear elastodynamics system for a class of strain energies}

\author{S\'ebastien Court\thanks{Institute for Mathematics and Scientific Computing, Karl-Franzens-Universit\"{a}t, Heinrichstr. 36, 8010 Graz, Austria, email: {\tt sebastien.court@uni-graz.at}.} \and Karl Kunisch\thanks{Institute for Mathematics and Scientific Computing, Karl-Franzens-Universit\"{a}t, Heinrichstr. 36, 8010 Graz, Austria, and Radon Institute, Austrian Academy of Sciences, email: {\tt karl.kunisch@uni-graz.at}.}}

\begin{document}

\maketitle

\begin{abstract}
The aim of this paper is to prove the existence of almost global weak solutions for the unsteady nonlinear elastodynamics system in dimension $d=2$ or $3$, for a range of strain energy density functions satisfying some given assumptions. These assumptions are satisfied by the main strain energies generally considered. The domain is assumed to be bounded, and mixed boundary conditions are considered. Our approach is based on a nonlinear parabolic regularization technique, involving the $p$-Laplace operator. First we prove the existence of a local-in-time solution for the regularized system, by a fixed point technique. Next, using an energy estimate, we show that if the data are small enough,  bounded by $\varepsilon >0$, then the maximal time of existence does not depend on the parabolic regularization parameter, and the behavior of the lifespan $T$ is $\gtrsim \log (1/\varepsilon)$, defining what we call here {\it almost global existence}. The solution is thus obtained by passing this parameter to zero. The key point of our proof is due to recent nonlinear Korn's inequalities proven by Ciarlet \& Mardare in $\W^{1,p}$ spaces, for $p>2$.
\end{abstract}

\noindent{\bf Keywords:} Nonlinear elasticity, Elastodynamics system, Hyperelastic materials, Parabolic regularization, $p$-Laplacian, Hyperbolic PDE, Global weak solutions.\\
\hfill \\
\noindent{\bf AMS subject classifications (2010):} 74B20, 35L70, 35L53, 74H20, 35A01, 35D30, 35K92.  

\tableofcontents

\section{Introduction}

\subsection{The model}
The elastodynamics system we consider in this paper is a hyperbolic partial differential equation combined with boundary conditions -- when the domain has a boundary -- and initial conditions, whose unknown is the displacement inside a deformable body. We denote by $u(\cdot,t)$ the displacement field at time $t$ with respect to the reference configuration represented by a bounded domain $\Omega$ of $\R^d$ (d = $2$ or $3$). It is assumed to obey the laws of elasticity (see\cite{Marsden} or~\cite{Ciarlet} for instance). The density of the body in the reference configuration is denoted by $\rho$. It is positive, and for a sake of simplicity, we assume it to be constant. We further assume that the boundary of the domain is split into two parts denoted by $\Gamma_D$ and $\Gamma_N$. For mathematical convenience, we will consider that the displacement is null on $\Gamma_D$, and that the Lebesgue measure of this boundary is not equal to zero, namely: $|\Gamma_D | > 0$. \textcolor{black}{Nevertheless, the pure traction case can also be considered (see Remark~\ref{rktraction}).}

%The initial and boundary-value problem we are interested in is the classical elastodynamics system (see \cite{Ciarlet} for instance), with in particular the consideration of mixed Neumann-type and homogeneous Dirichlet conditions:
Mixed boundary conditions are considered on $\p \Omega = \Gamma_D \sqcup \Gamma_N$. A homogeneous Dirichlet condition is imposed on $\Gamma_D$, and a non-homogeneous Neumann-type boundary condition is considered on $\Gamma_N$. For $0<T \leq \infty$, the system governing the evolution of the displacement $u$ is the following:
\begin{eqnarray} \label{mainsys}
\left\{\begin{array} {rcl}
	\displaystyle \rho \ddot{u} - \divg ((\I + \nabla u) \Sigma(u)) = f & &  \text{in } \Omega \times (0,T), \\
	u = 0, & &  \text{on } \Gamma_D\times (0,T), \\
	(\I + \nabla u) \Sigma(u)n = g & &  \text{on } \Gamma_N  \times (0,T),\\
	%\det(\I + \nabla u) > 0 & &  \text{on } \overline{\Omega} \times (0,T),\\
	\displaystyle u(\cdot,0) = u_0, \quad \dot{u}(\cdot,0) = u_1 & & \text{in } \Omega.
\end{array} \right.
\end{eqnarray}
In this system, the symbol $\I$ denotes the identity matrix of $\R^{d\times d}$, and $\Sigma$ denotes the so-called second Piola-Kirchhoff stress tensor, namely the derivative of the strain energy density function $\mathcal{W}$ with respect to the Green--St-Venant strain tensor $E$:
\begin{eqnarray*}
\Sigma(u)  =  \frac{\p \mathcal{W}}{\p E}(E(u)), & & E(u) = \frac{1}{2}\left((\I + \nabla u)^T(\I+\nabla u) - \I  \right).
\end{eqnarray*}
For the choice of the strain energy, a classical example is given by the St-Venant--Kirchhoff model, for which 
\begin{eqnarray*}
\mathcal{W}(E) & = & \mu_L \trace(E^2) + \frac{\lambda_L}{2} \trace(E)^2,
\end{eqnarray*}
where $\mu_L$ and $\lambda_L$ denote the classical Lam\'e coefficients. The functions $u_0$, $u_1$, $f$ and $g$ are data of the problem.

\subsection{Main result}
The question of local-in-time existence for the elastodynamics system has been first addressed in~\cite{HuKaMa1977}, for data reduced to initial conditions, and then in~\cite{Tougeron1988} for small Neumann data, both in the framework of strong solutions. A negative answer about the question of global existence has been given in~\cite{Knops1979}, and blow-up of strong solutions has been proven in~\cite{John1984} and~\cite{Gawinecki2008}, under particular assumptions on the strain energy density function. The global existence of large rigid displacements (but in the context of linearized elasticity) has been obtained in~\cite{Grandmont2002} for small data, and in~\cite{grandmont2007} for small strains. Almost global existence, that is to say lifespan depending on the bounds on the data, has been obtained in~\cite{John1988}, \cite{Klainerman}, and in~\cite{Xin2008} for the St-Venant--Kirchhoff model, and in~\cite{Lei2015} in the context of incompressible materials. 

For incompressible materials, the literature for global existence is abundant. Let us mention, for instance, the works of~\cite{Ebin1993, Ebin1996, Becca-thesis2003, Becca2005, Becca2007, Lei2016}, in the case of small data, and more recently in Eulerian formulation the results of \cite{Yin2016}, and~\cite{LeiWang} for Hookean elasticity. Finally, we mention the paper of~\cite{Zhang2009} where a locally distributed dissipation is added to the model, in order to stabilize the system.

The difficult question of global existence has been addressed in~\cite{Sideris1996, Sideris1996bis, Agemi} under the so-called {\it null condition}, and in~\cite{Sideris2000} under a {\it nonresonance} assumption for the stored energy.

As far as we know, regarding the class of general strain energies we consider in this article, no result concerning the existence of almost global solutions for the elastodynamics system has been obtained until now. Besides the complexity due to the nonlinearity of this system, the main difficulty lies in the control, by the total strain energy, of the gradient of the displacement. This difficulty can be now addressed thanks to the recent nonlinear Korn's inequalities proved in~\cite{Ciarlet2015}, and also in~\cite{Musesti2016}. More specifically, if $\det(\I+\nabla u) > 0$ almost everywhere in $\Omega$, the inequality given in the part~$(b)$ of Theorem~3 of~\cite{Ciarlet2015} yields in particular that
\begin{eqnarray*}
	\| u \|^p_{\left[\W^{1,p}(\Omega)\right]^d} & \leq & C \| E(u) \|^{p/2}_{[L^{p/2}(\Omega)]^{d\times d}},
\end{eqnarray*}
where $p>2$, and where the constant $C>0$ does not depend on $u \in \left[\W^{1,p}(\Omega)\right]^d$. In the example of the St-Venant--Kirchhoff model, the total strain energy on $\Omega$ controls the $\L^2$-norm of the tensor $E$, and thus for this case the exponent $p=4$ is well-chosen. This Korn's inequality is the key point leading to the main result of our work, namely:

\begin{theorem}
Let $\Omega$ be a bounded domain of $\R^d$ with $d =2$ or $3$. Assume that its boundary $\p \Omega = \Gamma_D \sqcup \Gamma_N$ is Lipschitz, and that $\Gamma_D$ is non-empty and relatively open in $\p \Omega$. Let be $p > 2$, $p \geq d$, and define $p' = p/(p-1)$. Assume that there exists $C>0$ such that, for all $E \in \displaystyle [\L^{p/2}(\Omega)]^{d\times d}$, the total strain energy satisfies
\begin{eqnarray*}
\int_{\Omega} \mathcal{W}(E) \,\d \Omega & \geq &  C\|E\|^{p/2}_{[ \L^{p/2}(\Omega)]^{d\times d}}.
\end{eqnarray*}
Assume further that $\mathcal{W}$ is of class $\mathcal{C}^1$ on $\left[\L^{p/2}(\Omega)\right]^{d\times d}$. Denoting by $\check{\Sigma}$ its differential and by $E$ the Green--St-Venant tensor, we assume that the tensor field $\check{\Sigma}\circ E$ is symmetric and locally $\alpha$-H\"{o}lderian on $\L^p\left(0,T;[\W^{1,p}]^d\right)$ for all $T>0$, with $\alpha = \min(1,(p-2)/2)$. 
Assume that $u_0 \in \displaystyle [\W^{1,p}(\Omega)]^d$, ${u_0}_{|\Gamma_D} \equiv 0$ and that $\det(\I+\nabla u_0) >0$ almost everywhere in $\Omega$. Let be $T>0$. Then there exists a constant $C >0$ independent of $T$ such that, if 
\begin{eqnarray*}
\int_\Omega \mathcal{W}(E(u_0))\,\d \Omega + \|u_1\|_{[\L^2(\Omega)]^d} + 
\| f \|_{\L^{2}(0,T;[\L^2(\Omega)]^d} +
\| g \|_{\L^{2}(0,T;[\H^{1/2}(\Gamma_N)']^d)} & \leq & C\exp(-CT), 
\end{eqnarray*}
then system~\eqref{mainsys} admits a solution $u$ such that
\begin{eqnarray*}
u \in \L^\infty(0,T;[\W^{1,p}(\Omega)]^d), 
\quad \dot{u} \in \L^{\infty}(0,T;[\L^{2}(\Omega)]^d),
\quad \ddot{u} \in \L^{p'}(0,T;[\W^{1,p}(\Omega)']^d).
\end{eqnarray*}
%if the invertibility condition $\det(\I + \nabla u) >0 $ is satisfied almost everywhere in $\Omega$, for all $t \geq 0$.
\end{theorem}

The assumption $p\geq d$ is made only for giving a sense in a time continuous space, namely $\mathcal{C}([0,\infty);\L^1(\Omega))$, to the quantity $\det(\I+\nabla u)$, whose the positivity is required for the Korn's inequality aforementioned. The smallness assumption on the data is made in order to take into account this criteria. We claim that if the condition $\det(\I+\nabla u) >0$ is assumed to be automatically satisfied, then regarding the steps of the proof of Theorem~\ref{thfinal}, the time of existence is actually $+\infty$, without smallness assumption on the data. This result corresponds to what we call {\it almost global existence} (we refer to~\cite{Cherrier}, page 154), even if because of this criteria the behavior of the life-span $T$ in function of the bound $\varepsilon$ on the data, namely $T \gtrsim \log (1/\varepsilon)$, is not as good as in the references cited above. The assumptions made on the strain energy in this theorem are actually satisfied by three important families of strain energies, namely the St-Venant--Kirchhoff model, the Fung's model (at least a polynomial approximation of this model), and the Ogden's model in some cases. See section~\ref{secmodels} for more details.

\subsection{Strategy}

The weak solution whose existence is proven in this paper is obtained by a parabolic regularization technique. The parabolic term we add to the elastodynamics system is the $p$-Laplace operator, in order to obtain the regularity of the time-derivative of the displacement in $[\W^{1,p}(\Omega)]^d$. The study of an evolutionary $p$-Laplace system enables us to define a mapping whose a fixed point is a weak solution of the so regularized elastodynamics system. For $T$ small enough, and under assumptions on the differential of the strain energy density function, by the Schauder's theorem we prove that this mapping admits a fixed point, and thus the existence of a local-in-time solution follows for the regularized elastodynamics system. Next, an estimate on the energy of the regularized system is obtained. Assuming that the total strain energy can control the norm of the Green--St-Venant tensor in $[\L^{p/2}(\Omega)]^{d\times d}$, we can thus control the gradient of the displacement in $[\L^p(\Omega)]^d$, thanks to the aforementioned nonlinear Korn's inequality. Furthermore, the energy estimate then shows that the maximal time of existence of the weak solution of the regularized system does not depend on the regularization coefficient, provided that the data are small enough. We can thus allow this parameter to tend to zero, and extract a solution by weak-* convergence. The solution so obtained is {\it a priori} not unique. The uniqueness could perhaps be proven in some particular cases, under some additional regularity property on the displacement. The question of uniqueness remains open in the general case.

The paper is organized as follows. The functional framework and notation are introduced in section~\ref{sec2}. In particular, assumptions are made on the type of strain energies we can consider in this paper, underlined by the study of classical examples. A preliminary result lies in the study of an evolutionary $p$-Laplace system in section~\ref{sec2para}. Section~\ref{sec3} is devoted to the proof of the existence of a local-in-time weak solution for the regularized elastodynamics system by a fixed point method. The question of global existence for system~\eqref{mainsys} is addressed in section~\ref{sec4}.
%For smoother data, higher regularity results are deduced in section~\ref{sec5}.

\section{Preliminaries} \label{sec2}
In the whole paper, we denote by $\Omega$ a bounded domain with Lipschitz boundary of $\R^d$, $d = 2$ or $3$. We denote by $\Gamma_D$ a non-empty relatively open subset of $\p \Omega$, and we define $\Gamma_N := \p \Omega \setminus \Gamma_D$. The first and second time-derivatives of a vector field $v$ will be denoted by $\dot{v}$ and $\ddot{v}$, respectively. For all $q >1$, we denote by $q' = q/(q-1)$ its conjugate number.

\subsection{Functional settings}
Throughout the paper, we will use the H\"{o}lder's inequality. Even if this is well-known, we recall it: Denoting by $\mathcal{S}$ a measure space, for $1\leq p,q \leq \infty$, and $1/r = 1/p + 1/q$, if $f \in \L^p(\mathcal{S})$ and $g \in \L^q(\mathcal{S})$, then
\begin{eqnarray*}
\| fg\|_{\L^r(\mathcal{S})} & \leq & \| f \|_{\L^p(\mathcal{S})} \| g \|_{\L^q(\mathcal{S})}.
\end{eqnarray*}
In particular, when the domain $\Omega$ is bounded, the embeddings $\L^r(\Omega) \hookrightarrow \L^s(\Omega)$ are continuous for $1 \leq s \leq r \leq \infty$. For $q > 1$, we consider multi-dimensional Sobolev spaces, by using the notation
\begin{eqnarray*}
\LL^q(\Omega) = \left[ \L^q(\Omega)\right]^d, & & 
\LLL^q(\Omega) = \left[ \L^q(\Omega)\right]^{d\times d},\\
\WW^q(\Omega) = \left[ \W^q(\Omega)\right]^d, & & 
\WWW^q(\Omega) = \left[ \W^q(\Omega)\right]^{d\times d}.
\end{eqnarray*}
The classical inner product for tensors in $\R^{d\times d}$ is denoted by $A:B = \trace(A^TB)$, and the associated norm is given by $|A|^2 = \trace(A^TA)$. Recall that it satisfies $|AB| \leq  |A| |B|$ for all $A, \ B \in \R^d$. For $p>1$, we define
\begin{eqnarray*}
\WW^{1,p}_{0,D}(\Omega) & := & \left\{ v\in \WW^{1,p}(\Omega), \ v_{|\Gamma_D} = 0 \right\}.
\end{eqnarray*}
As a specific norm for the space $\WW^{1,p}_{0,D}(\Omega)$, we define 
\begin{eqnarray*}
	\| v \|_{\WW^{1,p}_{0,D}(\Omega)} & : = & \| \nabla v \|_{\LLL^{p}(\Omega)},
\end{eqnarray*}
for all $v\in \WW^{1,p}_{0,D}(\Omega)$. Indeed, from the Rellich-Kondrachov theorem, the embedding $\WW^{1,p}(\Omega) \hookrightarrow \LL^p(\Omega)$ is compact, and since the operator $v \mapsto \nabla v$ is injective from $\WW^{1,p}_{0,D}(\Omega)$ to $\LLL^p(\Omega)$, the Petree-Tartar lemma (see \cite{Ern}, Lemma~A.38 page~469) applied to the equality $	\| v \|_{\WW^{1,p}(\Omega)}  =  
\| v \|_{\LL^{p}(\Omega)} + \| \nabla v \|_{\LLL^{p}(\Omega)}$ enables us to endow the space $\WW^{1,p}_{0,D}(\Omega)$ with the norm given above.

Finally, for the boundary $\Gamma_N$ we recall the trace inequality
\begin{eqnarray}
	\| v \|_{\WW^{1-1/p,p}(\Gamma_N)} & \leq & C_{p,N} \| v \|_{\WW^{1,p}(\Omega)}, \label{ineqtrace}
\end{eqnarray}
where the constant $C_{p,N} > 0$ does not depend on $v$. For the sake of brevity, we will use the notation
\begin{eqnarray*}
	\VV^{p}(\Omega) = \WW^{1,p}_{0,D}(\Omega), & & \VV^{p}(\Gamma_N) = \WW^{1-1/p,p}(\Gamma_N).
\end{eqnarray*}

\paragraph{Coerciveness of the Green - St-Venant strain tensor} 
The Green -- St-Venant strain tensor is defined by
\begin{eqnarray}
	E(v) & = & \frac{1}{2} \left( (\I+\nabla v)^T(\I +\nabla v) - \I \right) 
	= \frac{1}{2} \left( \nabla v + \nabla v^T + \nabla v^T \nabla v\right). \label{defE}
\end{eqnarray}
Part~$(b)$ of Theorem~3 of~\cite{Ciarlet2015} is a Korn type inequality for this nonlinear tensor. It provides coerciveness for this tensor in the space $\LLL^p(\Omega)$, with respect to $\VV^{p}(\Omega)$. In particular, given $p>2$, there exists a positive constant $C_K>0$ such that, for all $v \in \VV^{p}(\Omega)$ satisfying $\det(\I+\nabla v) >0$ almost everywhere in $\Omega$, the following inequality holds
\begin{eqnarray}
	\| v \|^2_{\VV^{p}(\Omega)} \leq \| v \|^2_{\WW^{1,p}(\Omega)} & \leq & 
	C_K\| E(v) \|_{\LLL^{p/2}(\Omega)}. \label{ineqKorn}
\end{eqnarray}
On the other hand, for $v_1, \ v_2 \in \WW^{1,p}(\Omega)$, with the Cauchy-Schwarz inequality it is easy to get the estimate
\begin{eqnarray}
	\| E(v)  \|_{\LLL^{p/2}(\Omega)} & \leq & 
	C\left(1 + \| \nabla v\|_{\LLL^p(\Omega)}  \right)
	\| \nabla v \|_{\LLL^p(\Omega)}. \label{ineqLipE} \label{estEnabla}
\end{eqnarray}
Here, as in the rest of the paper, the notation $C$ will define a generic positive constant, independent of $T$, the unknowns and the data of the problem, except $u_0$. But it may depend on $\Omega$, $\Gamma_D$, $\Gamma_N$, $\|u_0\|_{\VV^p(\Omega)}$, $\int_{\Omega}\mathcal{W}(E(u_0))\,\d \Omega$ and $\frac{\p \mathcal{W}}{\p E}(0)$.

\begin{remark} \label{rktraction}
	In the case where $\Gamma_D$ is reduced to $\emptyset$, namely the pure traction case, one can replace the space $\VV^p(\Omega)$ given above by the quotient space $\WW^{1,p}(\Omega) / SE(d)$, where $SE(d)$ denotes the special Euclidean group in $\R^d$, made of rigid displacements. The space $\VV^p(\Gamma_N)$ would be replaced by $\WW^{1-1/p,p}(\Gamma_N) / SE(d)$. From Part~$(b)$ of Theorem~1 of~\cite{Ciarlet2015}, we have the following nonlinear Korn's inequality:
	\begin{eqnarray*}
	\| v \|_{\WW^{1,p}(\Omega) / SE(d)}^2 = 
	\inf_{\mathbf{R}\in \mathbb{O}^+(d)} \| \mathbf{R}\nabla v\|_{\LLL^p(\Omega)}
	& \leq & C_K\| E(v) \|_{\LLL^{p/2}(\Omega)}.
	\end{eqnarray*}
	We have denoted by $\mathbb{O}^+(d)$ the rotation group in dimension $d$. Then the whole framework in the rest of the article would consist in looking for solutions to system~\eqref{mainsys} modulo rigid displacements.
	%Note in that case that the right-hand-sides $f$ and $g$ in system~\eqref{mainsys} would have to satisfy compatibility conditions.
\end{remark}

\subsection{Assumptions on the strain energy density function}
Let $p > 2$. The strain energy
\begin{eqnarray*}
\mathcal{W}: \ \LLL^{p/2}(\Omega) & \rightarrow & \L^1(\Omega)
\end{eqnarray*}
is a positive function of the Green -- St-Venant strain tensor $E$. When this mapping is G\^{a}teaux-differentiable, the derivative of $\mathcal{W}\circ E$ with respect to the displacement $u$ in the direction $v$ can be expressed as
\begin{eqnarray*}
\frac{\p (\mathcal{W}\circ E)}{\p u}.v & = & \frac{\p \mathcal{W}}{\p E}(E(u)) : (E'(u).v)
\end{eqnarray*}
with $E'(u).v = \frac{1}{2}((\I+\nabla u)^T\nabla v + \nabla v^T(\I+\nabla u)^T)$. If furthermore the mapping $\frac{\p \mathcal{W}}{\p E}(E(u))$ defines a symmetric tensor, then this expression reduces to
\begin{eqnarray*}
\frac{\p (\mathcal{W}\circ E)}{\p u}.v & = & (\I + \nabla u) \frac{\p \mathcal{W}}{\p E}(E(u)) : \nabla v.
\end{eqnarray*}
We assume the following set of hypotheses on the strain energy:
\begin{description}
	\item[A1]
	There exists $C > 0$ such that for all $E \in \LLL^{p/2}(\Omega)$ we have
	\begin{eqnarray}
	 C\|E\|^{p/2}_{\LLL^{p/2}(\Omega)} & \leq & \int_{\Omega}\mathcal{W}(E)\,\d\Omega. \label{ineqA1}
	\end{eqnarray}
	
	\item[A2] The strain energy $\mathcal{W}$ is of class $\mathcal{C}^1$ on $\LLL^{p/2}(\Omega)$. We denote
	\begin{eqnarray*}
		\check{\Sigma}(E) = \frac{\p \mathcal{W}}{\p E}(E) & \in & 
		\mathcal{L}\left(\LLL^{p/2}(\Omega); \L^{1}(\Omega)  \right) \simeq
		 \LLL^{(p/2)'}(\Omega) =  \LLL^{p/(p-2)}(\Omega). 
	\end{eqnarray*}
	 We assume that, for each symmetric tensor $E$, the tensor $\check{\Sigma}(E)$ is symmetric. When the tensor $E$ is expressed as function of a vector field $v$, through the expressions~\eqref{defE}, we will denote $\Sigma(v) := \check{\Sigma}(E(v))$.%The convexity of $\mathcal{W}\circ E$ then implies that for $v_1, \ v_2 \in \VV^p(\Omega)$ we have
	%\begin{eqnarray}
	%\left( (\I + \nabla v_1)\Sigma(v_1) - (\I + \nabla v_2)\Sigma(v_2) \right):\left( \nabla v_1 - \nabla v_2 \right) & \geq & 0. \label{ineqA2}
	%\end{eqnarray}
	
	\item[A3]
	The mapping $\mathcal{W}$ is of class $\mathcal{C}^1$ on $\L^{p/2}(0,T;\LLL^{p/2}(\Omega))$. Moreover, the mapping $\Sigma = \frac{\p \mathcal{W}}{\p E}\circ E$ is locally $\alpha$-sublinear with $ \displaystyle \alpha  = \min(1,(p-2)/2)$. More precisely, for $T>0$ and $R(T)>0$, there exists a positive constant $C_{R(T)} >0$ such that
	\begin{eqnarray}
	\| v \|_{\L^p(0,T;\VV^p(\Omega))} \leq R(T) 
	 & \Rightarrow & \| \Sigma(v)  \|_{\L^{(p/2)'}(0,T;\LLL^{(p/2)'}(\Omega))} \leq C + C_{R(T)}
	\| v \|^{\alpha}_{\L^p(0,T;\VV^p(\Omega))}. \label{ineqA3}
	\end{eqnarray}
	Here $R$ and $C_R$ are assumed to be non-decreasing with respect to $T$ and $R$, respectively.
	%where the constant $C_R >0$ is non-decreasing with respect to $R$.%We also assume that $\Sigma(0) = 0$.
	\end{description}
	
	\noindent Inequalities~\eqref{ineqKorn} and~\eqref{estEnabla} imply that
	\begin{eqnarray*}
	\| v \|^2_{\L^p(0,T;\VV^p(\Omega))} & \leq & C_K \| E(v) \|_{\L^{p/2}(0,T;\LLL^{p/2}(\Omega))}, \\
	\|E(v)  \|_{\L^{p/2}(0,T;\VV^{p/2}(\Omega))} & \leq & C\left(T^{1/p} + \| v \|_{\L^p(0,T;\VV^p(\Omega))}  \right)
	\| v \|_{\L^p(0,T;\VV^p(\Omega))},
	\end{eqnarray*}
	for $v \in \L^p(0,T;\VV^p(\Omega))$. Therefore, for $T \leq 1$ for instance (this assumption will be used only for proving the local-in-time result), assumption~{\bf A3} is implied by the following one. 
	
	\begin{description}
	\item[A3']
	The mapping $\mathcal{W}$ is of class $\mathcal{C}^1$ on $\LLL^{p/2}(\Omega)$. Moreover, its derivative $\check{\Sigma}$ is locally $\alpha$-sublinear on $\VV^p(\Omega)$, with $ \displaystyle \alpha  = \min(1,(p-2)/2)$. Namely, for $R>0$, there exists a positive constant $\check{C}_R >0$, non-decreasing with respect to $R$ such that
	\begin{eqnarray}
	\| E \|_{\LLL^{p/2}(\Omega)} \leq R 
	 & \Rightarrow & \| \check{\Sigma}(E)  \|_{\LLL^{(p/2)'}(\Omega)} \leq C + \check{C}_R
	\| E\|^{\alpha}_{\LLL^{p/2}(\Omega)}.
	\label{ineqA3prime}
	\end{eqnarray}
\end{description}

\noindent Finally, we sum up the hypotheses we make on the other data:
	\begin{description}
	\item[A4] We assume that $u_0 \in \mathbf{V}^p(\Omega)$ satisfies $\det(\I+\nabla u_0) >0$ almost everywhere in $\Omega$, and that
	\begin{eqnarray*}
		\int_{\Omega} \mathcal{W}(E(u_0))\,\d \Omega < \infty, 
		\quad u_1 \in \LL^2(\Omega),
		\quad f \in \L^{2}_{loc}(0,\infty;\LL^2(\Omega)), 
		\quad  g \in \L^{2}_{loc}(0,\infty;\VV^{2}(\Gamma_N)').
	\end{eqnarray*}
\end{description}

Assumptions~$\mathbf{A1}$ and~$\mathbf{A2}$ are used in an essential manner in section~\ref{sec4} for energy estimates, as well as assumption~$\mathbf{A4}$. Assumption~$\mathbf{A3}$ is mainly used in section~\ref{sec3}, in the proof of the fixed point method.

\begin{remark}
The regularity on $\Sigma$ postulated in assumption $\mathbf{A3}$ is made in particular in order to have $(\I+\nabla v)\Sigma(v)$ in $\L^{p'}(0,T;\LLL^{p'}(\Omega))$ for $v \in \VV^{p}(\Omega)$ (see Lemma~\ref{welldef}). 
%Indeed, since $\check{\Sigma}(E(0)) = 0$, from the H\"{o}lder's inequality combined with the inequalities~\eqref{ineqKorn} and~\eqref{ineqLipE}, we can estimate
%\begin{eqnarray*}
%\| (\I+\nabla v)\Sigma(E(v)) \|_{\LLL^{p'}(\Omega)} & \leq & 
%C \| (\I + \nabla v)\|_{\LLL^{p}(\Omega)} 
%\| \Sigma(E(v)) \|_{\LLL^{p/(p-2)}(\Omega)} \\
%& \leq & CC_R \left( 1+ \|  \nabla v\|_{\LLL^{p}(\Omega)} \right) \|E(v)\|^{(p-2)/2}_{\LLL^{p/2}(\Omega)} \\
%& \leq & CC_R \left( 1+ \|  \nabla v\|_{\LLL^{p}(\Omega)} \right)^{p/2} 
%\|\nabla v\|^{(p-2)/2}_{\LLL^{p}(\Omega)} \\
%& \leq & CC_R\left(\|\nabla v\|^{(p-2)/2}_{\LLL^{p}(\Omega)} +
%\|\nabla v\|^{p-1}_{\LLL^{p}(\Omega)} \right).
%\end{eqnarray*}
The nonlinear Korn's inequality given in~\cite{Mardare2004} in the case $p=2$ would enable us only to consider $\Sigma(v)$ in the space $\LLL^{\infty}(\Omega)$, which is not appropriate in view of the standard examples of strain energies, and leads to difficulties due to lack of reflexivity.
\end{remark}

\begin{remark}  \label{rk2A}
%First, the hypothesis $\check{\Sigma}(0) = 0$ of assumption~$\mathbf{A3}$ corresponds to most of the models of strain energies (see the examples below). 
In the assumptions~$\mathbf{A3}$ and~$\mathbf{A3'}$, we distinguish two cases. For the sake of simplicity, let us focus our comments on assumption~$\mathbf{A3'}$. First, when $ p \leq 4$, that is to say $\alpha = (p-2)/2$, inequality~\eqref{ineqA3prime} of assumption~{\bf A3'} is implied by
\begin{eqnarray*}
\| \check{\Sigma}(E) - \check{\Sigma}(0) \|^{p/(p-2)}_{\LLL^{p/(p-2)}(\Omega)} 
& \leq & \check{C}_{R}
\| E(v) \|^{p/2}_{\LLL^{p/2}(\Omega)},
\end{eqnarray*}
and even more so it is satisfied when $\check{\Sigma}$ is assumed to be locally $(p-2)/2$-H\"{o}lderian. Here we use that $(p/2)/(p/2)' = \alpha$. Secondly, when $p\geq 4$, and when the mapping $\check{\Sigma}$ is of class $\mathcal{C}^1$, with 
	\begin{eqnarray*}
	\frac{\p \check{\Sigma}}{\p E}(E) & \in & 
	\mathcal{L}\left( \LLL^{p/2}(\Omega) ; \LLL^{(p/2)'}(\Omega) \right) 
	\simeq \left[\LLL^{p/(p-4)}(\Omega)\right]^{d\times d},
	\end{eqnarray*}
	from the mean value theorem the mapping $\check{\Sigma}$ is locally Lipschitz, and thus this assumption is automatically satisfied:%, with $C = \| \check{\Sigma}(0)\|_{\LLL^{p/2)'}(\Omega)}$:
	\begin{eqnarray*}
		\| \check{\Sigma}(E)  \|_{\LLL^{(p/2)'}(\Omega)} & \leq & \| \check{\Sigma}(0)\|_{\LLL^{p/2)'}(\Omega)} +
		\sup_{\|E\|_{\LLL^{p/2}(\Omega)} \leq R} \left(\left\| \frac{\p \check{\Sigma}}{\p E}(E) \right\|_{\left[\LLL^{p/(p-4)}(\Omega)\right]^{d\times d}} \right)
		\| E\|_{\LLL^{p/2}(\Omega)}.
	\end{eqnarray*}
	%In the second case, when $(p-2)/2 < 1$, the derivative of $\check{\Sigma}$ (when it exists) is not exploitable as previously, and we could assume that $\check{\Sigma}$ is only $(p-2)/2$-H\"{o}lderian.
\end{remark}

\subsection{Examples of strain energy density functions} \label{secmodels}
Let us mention some models of strain energy, and see if the assumptions $\mathbf{A1}$ -- $\mathbf{A3}$ are satisfied for these examples. We refer to~\cite{Ciarlet} (section 4.10, page 183) or~\cite{Fung} for more comments on the models addressed below. Implicitly, in the expressions below we assume that the tensor $E$ is symmetric.

\paragraph{The St-Venant -- Kirchhoff model.}
It corresponds to the following strain energy
\begin{eqnarray*}
	\mathcal{W}_1(E) & = & \mu_L \trace\left( E^2 \right) + \frac{\lambda_L}{2} \trace(E)^2,
\end{eqnarray*}
where $\mu_L >0$ and $\lambda_L \geq 0$ are the so-called Lam\'e coefficients. Here, the exponent $p = 4$ is well-fitted, because in this case $p/2 = p/(p-2) = 2$, and we can estimate easily
\begin{eqnarray*}
\int_{\Omega} \mathcal{W}_1(E)\, \d \Omega & \geq & C \| E \|^2_{\LLL^{2}(\Omega)}, \\
\check{\Sigma}_1(E) := \frac{\p \mathcal{W}_1}{\p E}(E) & = & 2\mu_L E + \lambda_L \trace(E)\I, \\
\|\check{\Sigma}_1(E) \|_{\LLL^2(\Omega)} & \leq & C\|E \|_{\LLL^2(\Omega)}.
\end{eqnarray*}
Thus the assumptions $\mathbf{A1}$ -- $\mathbf{A3}$ (and even $\mathbf{A3'}$) are verified for this example. 

\paragraph{The Fung's model.}
It corresponds to the following strain energy
\begin{eqnarray*}
	\mathcal{W}_2(E) & = & \mathcal{W}_2(0) + \beta \left(\exp\left(\gamma \ \trace(E^2)\right) - 1\right),
\end{eqnarray*}
where $\mathcal{W}_2(0) \geq 0$, $\beta >0$ and $\gamma > 0$ are given coefficients. We only know that the space $\W^{s,q}(\Omega)$ is invariant under composition of the exponential function if $s \geq 1$, for certain values of $q$ (see~\cite{BB1974}, Lemma~A.2. page 359). Therefore, in our context where $E$ is considered only in $\LLL^{p/2}(\Omega)$, we need to simplify this model. We approximate this energy by the following one
\begin{eqnarray*}
	\mathcal{W}^N_2(E) & = &  \mathcal{W}_2(0) + \beta \sum_{k = 1}^N \frac{\gamma^k \trace(E^2)^k}{k!},
\end{eqnarray*}
where $2 \leq N \in \N$ is the degree of approximation of the power series defining the exponential function. Choosing $p =4N$, we have $(p-2)/2 \geq 1$, $p/(p-4) = N/(N-1)$ and the following estimate holds:
\begin{eqnarray*}
	\int_{\Omega} \mathcal{W}^N_2(E)\,\d \Omega & \geq & %\mathcal{W}_2(0) + 
	C\| E\|^{2N}_{\LLL^{2N}(\Omega)}.
\end{eqnarray*}
From the identities
\begin{eqnarray*}
	\check{\Sigma}^N_2(E) := \frac{\p \mathcal{W}^N_2}{\p E}(E) & = &
	2\beta \gamma\left(\sum_{k=0}^{N-1} \frac{\gamma^{k} \trace(E^2)^k}{k!}\right)E, \\
	\frac{\p \check{\Sigma}^N_2}{\p E}(E) &  = & 2\beta \gamma\left(\sum_{k=0}^{N-1} \frac{\gamma^{k} \trace(E^2)^k}{k!}\right)\I + 4\beta\gamma^2 \left(\sum_{k=0}^{N-2} \frac{\gamma^{k} \trace(E^2)^k}{k!}\right)E\otimes E \quad 
	\text{in } \R^{d\times d \times d\times d}
\end{eqnarray*}
(where $(E\otimes E)F = (E:F)E$ as a notation, and $\I$ denoting also the identity mapping of $\R^{d^4}$), and from Remark~\ref{rk2A}, for $R>0$ large enough we can estimate
\begin{eqnarray*}
\left\| \frac{\p \check{\Sigma}^N_2}{\p E} \right\|_{\L^{N/(N-1)}\left(0,T;\left[\LLL^{N/(N-1)}(\Omega)\right]^{d\times d}\right)} & \leq &
C \sum_{k=0}^{N-1} \|E\|^{2k}_{\L^{2N}(0,T;\LLL^{2N}(\Omega))} 
\leq C \sum_{k=0}^{N-1} R^{2k} \leq CR^{2N-2}, \\
		\| \check{\Sigma}_2^N(E)  \|_{\LLL^{2N/(2N-1)}(\Omega)} & \leq & 
		\| \check{\Sigma}_2^N(0)  \|_{\LLL^{2N/(2N-1)}(\Omega)} + 
		CR^{2N-2}
		\| E \|_{\LLL^{2N}(\Omega)}.
\end{eqnarray*}
Assumptions $\mathbf{A1}$ -- $\mathbf{A3'}$ are thus verified for this approximation of the Fung's model.

\paragraph{The Ogden's model.}
The family of strain energies corresponding to this model are linear combinations of energies of the following form
\begin{eqnarray*}
	\mathcal{W}_3(E) & = &  \trace\left((2E+\I)^{\gamma} - \I \right),
\end{eqnarray*}
where $\gamma \in \R$. Since the tensor $2E+\I$ is real and symmetric, the expression $(2E+\I)^{\beta} $ makes sense for all number $\beta \in \R$, and the energy $\mathcal{W}_3(E)$ can be expressed in terms of the eigenvalues of $2E+\I$. This general form of the strain energy includes the cases of the Neo-Hookean and Mooney-Rivlin models ($\gamma = 1$ and $\gamma \in \{-1, +1\}$ respectively). But here we only evoke the case $\gamma > 1$. First, since $2E(u)+\I = (\I+\nabla u)^T(\I+\nabla u)$, if $(\lambda_i)_{1\leq i \leq d}$ denote the singular values of $\I+\nabla u$, and $(\mu_i)_{1\leq i \leq d}$ denote those of $E(u)$, we have
\begin{eqnarray*}
\trace\left((2E+\I)^{\gamma} -\I \right)  =  \sum_{i=1}^d \left(\lambda_i^{2\gamma}  -1\right)
= \sum_{i=1}^d \left((1+2\mu_i)^\gamma -1\right)
 \geq  \sum_{i=1}^d (2\mu_i)^\gamma 
& \geq & C\left(\sum_{i=1}^d \mu_i^2\right)^{\gamma/2}
= C\left| E \right|^{\gamma},
\end{eqnarray*}
because of the equivalence of norms in $\R^d$. Thus, by choosing $p = 2\gamma >2$, we have
\begin{eqnarray*}
\int_{\Omega} \mathcal{W}_3(E)\, \d \Omega  =
\left\| 2E+\I \right\|_{\LLL^{p/2}(\Omega)}^{p/2} - d|\Omega| & \geq & C\|E\|_{\LLL^{p/2}(\Omega)}^{p/2} , 
\end{eqnarray*}
that is to say~$\mathbf{A1}$ holds, and~$\mathbf{A2}$ can be easily checked. For the assumptions~$\mathbf{A3}$ and~$\mathbf{A3'}$, since the derivative of $\mathcal{W}_3$ is given by $\check{\Sigma}_3(E) = 2\gamma (2E+\I)^{\gamma-1}$, the case $\gamma \geq 2$ can be treated as previously. Turning to $1 < \gamma < 2$, we have that $(p-2)/2 = \gamma-1 \in (0,1)$. Since the derivative of $\mathcal{W}_3$ writes $\check{\Sigma}_3(E) = 2\gamma (2E+\I)^{\gamma-1}$, and since the function $x \mapsto x^{\gamma-1}$ is $(\gamma-1)$-H\"{o}lderian on $[0,1]$, we deduce that for $E$ small enough in $\L^{\gamma}(0,T;\LLL^{\gamma}(\Omega))$, namely $R(T)$ small enough, the following estimate holds
\begin{eqnarray*}
\|\check{\Sigma}_3(E)\|
_{\L^{\gamma/(\gamma-1)}(0,T;\LLL^{\gamma/(\gamma-1)}(\Omega))} 
& \leq & 
T^{1-1/\gamma}\|\check{\Sigma}_3(0)\|
_{\L^{\gamma/(\gamma-1)}(0,T;\LLL^{\gamma/(\gamma-1)}(\Omega))} +
C_{R(T)} \| E \|_{\L^{\gamma}(0,T;\LLL^\gamma(\Omega))}^{\gamma-1}.
\end{eqnarray*}
We will see in section~\ref{sec3} (more precisely Lemma~\ref{lemmadefint}) that the constant $R(T)$ of assumption~$\mathbf{A3}$ can be chosen small enough, provided that $T$ is chosen small enough. Thus, for the Ogden's model with a coefficient $\gamma >1$, assumptions $\mathbf{A1}$--$\mathbf{A3}$ are satisfied.

\section{A nonlinear parabolic system} \label{sec2para}
In this section, we are interested in the following nonlinear parabolic system
\begin{eqnarray} \label{parasys}
\left\{\begin{array} {rcl}
\displaystyle \rho \dot{w} - \kappa \divg (|\nabla w |^{p-2} \nabla w) = f & &  \text{in } \Omega \times (0,T), \\
w = 0 & &  \text{on } \Gamma_D\times (0,T), \\
\displaystyle \kappa|\nabla w |^{p-2} \nabla w\, n = g & &  \text{on } \Gamma_N \times (0,T),\\
\displaystyle w(\cdot,0) = u_1 & & \text{in } \Omega,
\end{array} \right.
\end{eqnarray}
where $\kappa > 0$, $p\geq 2$, and $T>0$ is fixed and arbitrary. System~\eqref{parasys} is an evolutionary $p$-Laplace equation with mixed boundary conditions. Throughout it is assumed that
\begin{eqnarray*}
u_1 \in \LL^2(\Omega), \quad f \in \L^{p'}(0,T; \VV^p(\Omega)'), \quad g \in \L^{p'}(0,T;\VV^p(\Gamma_N)').
\end{eqnarray*}

\begin{definition} \label{def1}
	%Let $p > 2$. Assume that $u_1 \in \LL^2(\Omega)$, $f \in \L^{p'}(0,T; \VV^p(\Omega)')$ and $g \in \L^{p'}(0,T;\VV^p(\Gamma_N)')$. 
	We say that $w$ is a weak solution of system~\eqref{parasys} if $w\in \L^{p}(0,T; \VV^p(\Omega))$, $\dot{w} \in \L^{p'}(0,T; \VV^p(\Omega)')$, $w(0) = u_1$, and for all $\varphi \in \L^p(0,T;\VV^p(\Omega))$ we have
	\begin{eqnarray*}
		\rho \langle \dot{w} ; \varphi \rangle_{\VV^p(\Omega)';\VV^p(\Omega)} + 
		\kappa \int_{\Omega}|\nabla w|^{p-2} \nabla w : \nabla \varphi\, \d \Omega & = &
		\langle f ; \varphi \rangle_{\VV^p(\Omega)';\VV^p(\Omega)} +
		\langle g ; \varphi \rangle_{\VV^p(\Gamma_N)';\VV^p(\Gamma_N)},
	\end{eqnarray*}
	almost everywhere in $(0,T)$.
\end{definition}

\begin{remark}
	By definition, a weak solution of system~\eqref{parasys} lies in the space $W(0,T;\WW^{1,p}(\Omega))$ defined by 
	\begin{eqnarray*}
		w \in W(0,T;\V^{p}(\Omega)) & \Leftrightarrow & 
		\left\{ \begin{array} {l}
			w \in \L^p(0,T;\VV^{p}(\Omega)) \\
			\displaystyle \dot{w} \in \L^{p'}(0,T;\VV^{p}(\Omega)')
		\end{array}\right. ,
	\end{eqnarray*}
	corresponding to the Gelfand triplet $V \hookrightarrow H \equiv H' \hookrightarrow V'$, with $H = \LL^2(\Omega)$, $V = \VV^{p}(\Omega)$, and dense embeddings. So it is well-known that such a solution lies also in $C([0,T];\LL^2(\Omega))$, and thus the space $\LL^2(\Omega)$ in which the initial condition is considered makes sense.
\end{remark}

\begin{proposition} \label{propLaplace}
	%Let $p \geq 2$. Assume that $u_1 \in \LL^2(\Omega)$, $f \in \L^{p'}(0,T;\VV^{p}(\Omega)')$ and $g \in \L^{p'}(0,T;\VV^{p}(\Gamma_N)')$. 
	System~\eqref{parasys} admits a unique weak solution $w$, in the sense of definition~\ref{def1}. Moreover, it satisfies the estimate
	\begin{eqnarray}
	\left\| \dot{w} \right\|^{p'}_{\L^{p'}(0,T;\VV^{p}(\Omega)')} + 
	\| w\|^p_{\L^p(0,T;\VV^p(\Omega))} 
	& \leq & 
	C \left(\|u_1\|^2_{\LL^2(\Omega)} \right. \nonumber \\ 
	& & + \left.\| f \|^{p'}_{\L^{p'}(0,T;\VV^{p}(\Omega)')}
	+ \| g \|^{p'}_{\L^{p'}(0,T;\VV^{p}(\Gamma_N)')}
	\right),
	\end{eqnarray}
	where the constant $C$ depends only on $\rho$, $\kappa$ and $\Omega$. 
\end{proposition}

\begin{proof}
	System~\eqref{parasys} is a $p$-Laplace evolution problem. Existence and uniqueness of a weak solution for this system have been proven in~\cite{BardosBrezis} for instance (see Theorem~V.3. p.~387)\footnote{Note that in this reference the framework would also enable us to consider a non-constant density $\rho$.}. Uniqueness is due to the convexity of the function $v\mapsto |v|^p$. Let us prove the announced estimate, which will be obtained with standard arguments. Taking the inner product of the first equation of~\eqref{parasys} by $w$ and integrating on $\Omega$ yields, with the Green formula
	\begin{eqnarray*}
		\frac{\rho}{2}\frac{\d }{\d t}\left( \| w \|^2_{\LL^2(\Omega)} \right) 
		+ \kappa \| \nabla w \|^p_{\LL^p(\Omega)} & = &
		\langle f ; w \rangle_{\VV^{p}(\Omega)' ; \VV^{p}(\Omega)} +
		\langle g ; w \rangle_{\VV^{p}(\Gamma_N)' ; \VV^{p}(\Gamma_N)}, \\
		& \leq & 
		\| f \|_{\VV^{p}(\Omega)'} \| w\|_{\VV^{p}(\Omega)} +
			\| g \|_{\VV^{p}(\Gamma_N)'} \| w \|_{\VV^{p}(\Gamma_N)}.
	\end{eqnarray*}
	Keep in mind the identity $\| \nabla w \|_{\LL^p(\Omega)} = \|  w \|_{\WW^{1,p}(\Omega)}$, and recall the trace inequality $\| w \|_{\VV^{p}(\Gamma_N)} \leq C_{p,N} \| w \|_{\VV^{p}(\Omega)}$, where $C_{p,N} > 0$ depends only on $\Gamma_N$, $\Omega$ and $p$. Next, integrating this inequality in time between $0$ and $T$ gives, with the Young's inequality involving some $\alpha > 0$,
	\begin{eqnarray*}
		\frac{\rho}{2} \| w(T) \|^2_{\LL^2(\Omega)} + 
		\kappa \int_0^T\| \nabla w(t) \|^p_{\LL^p(\Omega)} \d t
		& \leq & \frac{\rho}{2} \| u_1 \|^2_{\LL^2(\Omega)} + \\
		 & &  \frac{\alpha^{-p'}}{p'}\int_0^T \| f(t) \|^{p'}_{\VV^{p}(\Omega)'}\d t +
		\frac{\alpha^p}{p}\int_0^T \| w(t)\|^p_{\VV^{p}(\Omega)} \d t + \\
		& &  \frac{\alpha^{-p'}}{p'}\int_0^T \| g(t) \|^{p'}_{\VV^{p}(\Gamma_N)'}\d t
		+ \frac{(\alpha C_{p,N})^p}{p} \int_0^T \| w(t)\|^p_{\VV^{p}(\Omega)} \d t.  
	\end{eqnarray*}
	Choose $\alpha > 0$ small enough, such that $(\alpha ^p + (\alpha C_{p,N})^p)/p \leq \kappa/2$, and we obtain
	\begin{eqnarray*}
	%\frac{\rho}{2} \| w(T) \|^2_{\LL^2(\Omega)} + 
	\frac{\kappa}{2} \|  w \|^p_{\L^p(0,T;\VV^{p}(\Omega))}
	& \leq & \frac{\rho}{2} \| w(0) \|^2_{\LL^2(\Omega)} + C \left( 
	| f \|^{p'}_{\L^{p'}(0,T;\VV^{p}(\Omega)')}
	+ \| g \|^{p'}_{\L^{p'}(0,T;\VV^{p}(\Gamma_N)')}
	\right).
	\end{eqnarray*}
	For the estimate on the time-derivative, due to the equality $\rho \dot{w}  =  f + \kappa \divg\left(|\nabla w |^{p-2}\nabla w \right)$ it is sufficient to control the term $\divg\left(|\nabla w |^{p-2}\nabla w \right)$ in the space $\L^{p'}(0,T;\VV^{p}(\Omega)')$. First, it is easy to verify that $|\nabla w |^{p-2} \nabla w$ lies in $\L^{p'}(0,T;\LL^{p'}(\Omega))$. More specifically, we have 
	\begin{eqnarray*}
	\| |\nabla w |^{p-2} \nabla w \|^{p'}_{\L^{p'}(0,T;\LL^{p'}(\Omega))} & = & 
	\| \nabla w \|^p_{\L^{p}(0,T;\LL^{p}(\Omega))}, \\
	\| |\nabla w |^{p-2} \nabla w \|_{\L^{p'}(0,T;\LL^{p'}(\Omega))} & = & 
	\|  w \|^{p-1}_{\L^{p}(0,T;\VV^{p}(\Omega))}.
	\end{eqnarray*}
	Therefore, for all $\varphi \in \VV^{p}(\Omega)$, from H\"older's inequality we have
	\begin{eqnarray*}
	 \left\langle \divg\left(|\nabla w |^{p-2}\nabla w \right) ; \varphi \right\rangle_{\VV^{p}(\Omega)'; \VV^{p}(\Omega)}
	& = &  
	- \int_{\Omega} |\nabla w |^{p-2}\nabla w : \nabla \varphi\, \d \Omega 
	+  \langle g ; \varphi \rangle_{\VV^{p}(\Gamma_N)'; \VV^{p}(\Gamma_N)}, \\
	\left| \left\langle \divg\left(|\nabla w |^{p-2}\nabla w \right) ; \varphi \right\rangle_{\VV^{p}(\Omega)'; \VV^{p}(\Omega)} \right| & \leq &
	\| |\nabla w |^{p-2} \nabla w \|_{\LL^{p'}(\Omega))} \| \nabla \varphi \|_{\LL^{p}(\Omega))}
	+ C_{p,N}\| g\|_{\VV^{p}(\Gamma_N)'}\|\varphi \|_{\VV^{p}(\Omega)} \\
	& \leq & C \left( \|  w \|^{p-1}_{\VV^{p}(\Omega)} + 
	\| g\|_{\VV^{p}(\Gamma_N)'}
	\right)\|\varphi \|_{\VV^{p}(\Omega)},
	\end{eqnarray*}
	and thus, by Young's inequality
	\begin{eqnarray*}
	\left\| \divg\left(|\nabla w |^{p-2}\nabla w \right) \right\|^{p'}_{\VV^{p}(\Omega)'} 
	& \leq & C \left( \|  w \|^{p-1}_{\VV^{p}(\Omega)} + 
	\| g\|_{\VV^{p}(\Gamma_N)'}
	\right)^{p'} \\
	& \leq & C \left( \|  w \|^{p}_{\VV^{p}(\Omega)} + 
	\| g\|^{p'}_{\VV^{p}(\Gamma_N)'}
	\right),
	\\
	\left\| \divg\left(|\nabla w |^{p-2}\nabla w \right) \right\|^{p'}_{\L^{p'}(0,T;\VV^{p}(\Omega)')} 
	& \leq & C \left( \|  w \|^{p}_{\L^{p}(0,T;\VV^{p}(\Omega))} + 
	\| g\|^{p'}_{\L^{p'}(0,T;\VV^{p}(\Gamma_N)')}
	\right),
	\end{eqnarray*}
	which concludes the proof.
\end{proof}

We now consider system~\eqref{parasys} with particular additional right-hand-sides, namely
\begin{eqnarray} \label{parasys2}
\left\{\begin{array} {rcl}
\displaystyle \rho \dot{w} - \kappa \divg (|\nabla w |^{p-2} \nabla w) = f + \divg A & &  \text{in } \Omega \times (0,T), \\
w = 0 & &  \text{on } \Gamma_D\times (0,T), \\
\displaystyle \kappa|\nabla w |^{p-2} \nabla w\, n = g - An& &  \text{on } \Gamma_N \times (0,T),\\
\displaystyle w(\cdot,0) = u_1 & & \text{in } \Omega,
\end{array} \right.
\end{eqnarray} 
where $A$ is a given tensor field.

\begin{definition} \label{def2}
	%Let $p\geq 2$. Assume that $u_1 \in \LL^2(\Omega)$, $f \in \L^{p'}(0,T; \VV^p(\Omega)')$, $g \in \L^{p'}(0,T;\VV^p(\Gamma_N)')$ and 
	Let $A \in \L^{p'}(0,T; \LLL^{p'}(\Omega))$. We say that $w$ is a weak solution of system~\eqref{parasys2} if $w\in \L^{p}(0,T; \VV^p(\Omega))$, $\dot{w} \in \L^{p'}(0,T; \VV^p(\Omega)')$, $w(0) = u_1$, and for all $\varphi \in \L^p(0,T;\VV^p(\Omega))$ we have
	\begin{eqnarray*}
	\rho \langle \dot{w} ; \varphi \rangle_{\VV^p(\Omega)';\VV^p(\Omega)} + 
	\kappa \int_{\Omega}|\nabla w|^{p-2} \nabla w : \nabla \varphi\, \d \Omega & = &
	\langle f ; \varphi \rangle_{\VV^p(\Omega)';\VV^p(\Omega)} +
	\langle g ; \varphi \rangle_{\VV^p(\Gamma_N)';\VV^p(\Gamma_N)}\\
	& & - \int_{\Omega} A: \nabla \varphi\, \d \Omega,
	\end{eqnarray*}
	almost everywhere in $(0,T)$.
\end{definition}

From Proposition~\ref{propLaplace} we can deduce the following result.

\begin{corollary} \label{coroLaplace}
	%Let $p\geq 2$. Assume that $u_1 \in \LL^2(\Omega)$, $f \in \L^{p'}(0,T; \VV^p(\Omega)')$, $g \in \L^{p'}(0,T;\VV^p(\Gamma_N)')$ and 
	Let $A \in \L^{p'}(0,T; \LLL^{p'}(\Omega))$. Then system~\eqref{parasys2} admits a unique weak solution $w$, in the sense of Definition~\ref{def2}. Moreover, it satisfies the following estimate
	\begin{eqnarray}
	\left\| \dot{w} \right\|^{p'}_{\L^{p'}(0,T;\VV^{p}(\Omega)')} + 
	\| w\|^p_{\L^p(0,T;\VV^p(\Omega))} 
	& \leq & 
	C \left(\|u_1\|^2_{\LL^2(\Omega)}  \nonumber  + \| f \|^{p'}_{\L^{p'}(0,T;\VV^{p}(\Omega)')}
	+ \| g \|^{p'}_{\L^{p'}(0,T;\VV^{p}(\Gamma_N)')} \right. \\
	& & \left. +  \| A \|^{p'}_{\L^{p'}(0,T; \LLL^{p'}(\Omega))} \right), \label{estcoro}
	\end{eqnarray}
	where the constant $C$ depends only on $\rho$, $\kappa$ and $\Omega$. 
\end{corollary}

\begin{proof}
	In the variational formulation of system~\eqref{parasys2}, the Green's formula reduces the terms involving $A$ to
	\begin{eqnarray*}
	\langle \divg A ; \varphi \rangle_{\VV^p(\Omega)';\VV^p(\Omega)} 
	- \langle  An ; \varphi \rangle_{\VV^p(\Gamma_N)';\VV^p(\Gamma_N)} 
	& = & \int_{\Omega} A:\nabla \varphi \,\d \Omega.
	\end{eqnarray*}
	From the H\"{o}lder's inequality, we have
	\begin{eqnarray*}
		\left| \int_{\Omega} A:\nabla \varphi\, \d \Omega \right| & \leq & 
		\| A\|_{\LLL^{p'}(\Omega)} \| \varphi \|_{\VV^p(\Omega)},
	\end{eqnarray*}
	Therefore the result of Proposition~\ref{propLaplace} holds in this case. In particular, the steps of the proof of the announced estimate are the same as those given in the proof of Proposition~\ref{propLaplace}.
\end{proof}

\section{Local existence by parabolic regularization} \label{sec3}
We choose $\kappa > 0$, $p>2$, $f \in \L^{p'}(0,T;\VV^{p}(\Omega)')$, $g \in \L^{p'}(0,T;\VV^{p}(\Gamma_N)')$, and consider the following system
\begin{eqnarray} \label{mainsysreg}
\left\{\begin{array} {rcl}
\displaystyle \rho \ddot{u} - \kappa \divg \left( \left|\nabla \dot{u} \right|^{p-2} \nabla \dot{u}\right) - \divg ((\I + \nabla u) \Sigma(u)) = f & &  \text{in } \Omega \times (0,T), \\
u = 0 & &  \text{on } \Gamma_D\times (0,T), \\
\kappa \left|\nabla \dot{u} \right|^{p-2} \nabla \dot{u}\, n + (\I + \nabla u) \Sigma(u)\, n = g & &  \text{on } \Gamma_N \times (0,T),\\
%\det(\I + \nabla u) > 0 & &  \text{on } \overline{\Omega} \times (0,T),\\
\displaystyle u(\cdot,0) = u_0, \quad \dot{u}(\cdot,0) = u_1 & & \text{in } \Omega.
\end{array} \right.
\end{eqnarray}
The definition of a weak solution for this system is inspired by Definition~\ref{def2}, with $\dot{u}$ in the role of $w$, and $(\I + \nabla u) \Sigma(u)$ in the role of the tensor field $A$. 

\begin{definition} \label{def3}
	%Let $p\geq 2$. Assume that $u_0 \in \VV^p(\Omega)$, $u_1 \in \LL^2(\Omega)$, $f \in \L^{p'}(0,T; \VV^p(\Omega)')$ and $g \in \L^{p'}(0,T;\VV^p(\Gamma_N)')$. 
	We say that $w$ is a weak solution of system~\eqref{mainsysreg} if $w\in \L^{p}(0,T; \VV^p(\Omega))$, $\dot{w} \in \L^{p'}(0,T; \VV^p(\Omega)')$, $w(0) = u_1$, and for all $\varphi \in \L^p(0,T;\VV^p(\Omega))$ we have
	\begin{eqnarray*}
		\left\{\begin{array} {lcl}
		 \displaystyle u(\cdot,t) \ := \ u_0(\cdot) + \int_0^t w(\cdot,s) \d s, & & \\
		\displaystyle  \rho \langle \dot{w} ; \varphi \rangle_{\VV^p(\Omega)';\VV^p(\Omega)} + 
		\kappa \int_{\Omega}|\nabla w|^{p-2} \nabla w : \nabla \varphi\, \d \Omega 
		+ \int_{\Omega}(\I + \nabla u) \Sigma(u) : \nabla \varphi\, \d \Omega
		& = &
		\langle f ; \varphi \rangle_{\VV^p(\Omega)';\VV^p(\Omega)}  \\
		 & & + \langle g ; \varphi \rangle_{\VV^p(\Gamma_N)';\VV^p(\Gamma_N)},
		\end{array}\right.
	\end{eqnarray*}
	almost everywhere in $(0,T)$.
\end{definition}

We look for a weak solution of system~\eqref{mainsysreg} in the following set
\begin{eqnarray*}
\mathcal{B}_R(T) & = & \left\{ w\in \L^p(0,T;\VV^{p}(\Omega)), \  
\| w\|_{\L^p(0,T;\VV^{p}(\Omega))} \leq R \right\},
\end{eqnarray*}
where $R>0$ will be chosen large, and $T>0$ small enough. This weak solution can be seen as a fixed point of the following mapping
\begin{eqnarray*}
	\begin{array} {cccc}
		\mathcal{N} : & \mathcal{B}_R(T) & \rightarrow & \L^p(0,T;\LL^{p}(\Omega)) \\
		 & \tilde{w} & \mapsto & w
	\end{array}
\end{eqnarray*}
where $\tilde{w}$ defines
\begin{eqnarray*}
\tilde{u}(x,t) & := & u_0(x) + \int_0^t \tilde{w}(x,s) \d s, \qquad x \in \Omega,
\end{eqnarray*}
and $w$ is the solution -- in the sense of Definition~\ref{def2} -- of the following system
\begin{eqnarray}
\left\{\begin{array} {rcl}
	\displaystyle \rho \dot{w} - \kappa \divg (|\nabla w |^{p-2} \nabla w) = f + \divg \left((\I + \nabla \tilde{u}) \Sigma(\tilde{u})\right) & &  \text{in } \Omega \times (0,T), \\
	w = 0 & &  \text{on } \Gamma_D\times (0,T), \\
	\displaystyle \kappa|\nabla w |^{p-2} \nabla w\, n = g - (\I + \nabla \tilde{u}) \Sigma(\tilde{u})n & &  \text{on } \Gamma_N \times (0,T),\\
	\displaystyle w(\cdot,0) = u_1 & & \text{in } \Omega.
\end{array} \right. \label{sysptfix}
\end{eqnarray}

System~\eqref{sysptfix} is of the same type as system~\eqref{parasys2}. From Corollary~\ref{coroLaplace}, the solution of this system is well-defined in $\W(0,T;\WW^{1,p}(\Omega))$, provided that the right-hand-sides lie in the corresponding spaces, in particular $(\I + \nabla \tilde{u}) \Sigma(\tilde{u})$ should lie in $\L^{p'}(0,T;\LLL^{p'}(\Omega))$. This point is verified in Lemma~\ref{welldef} below.

\subsection{Stability estimates}

\begin{lemma} \label{lemmadefint}
Let be $R>0$, $0<T \leq 1$ and $w \in \mathcal{B}_R(T)$. Then the function defined by
\begin{eqnarray}
v(\cdot,t)  =  u_0(\cdot) + \int_0^t w(\cdot,s)\d s \label{defint}
\end{eqnarray}
satisfies
\begin{eqnarray}
\|v\|_{\L^p(0,T;\VV^p(\Omega))} & \leq & T^{1/p}(\|u_0\|_{\VV^p(\Omega)} + R). \label{estdefint}
\end{eqnarray}
\end{lemma}

\begin{proof}
From the triangle inequality leading to
\begin{eqnarray*}
	\|v(t)\|_{\L^p(0,T;\VV^p(\Omega))} & \leq & T^{1/p}\| u_0 \|_{\VV^p(\Omega)} + 
	\left\| \int_0^t w(s) \d s \right\|_{\L^p(0,T;\VV^p(\Omega))},
\end{eqnarray*}
we deduce with the H\"{o}lder's inequality
\begin{eqnarray*}
	\|v\|_{\L^p(0,T;\VV(\Omega))} & \leq & T^{1/p}\| u_0 \|_{\VV^p(\Omega)}
	+ \left(\int_0^T t^{p/{p'}} \| w \|^p_{\L^p(0,t;\VV^p(\Omega))}\d t \right)^{1/p}\\ 
	 & \leq &  T^{1/p}\| u_0 \|_{\VV^p(\Omega)} + T\| w \|_{\L^p(0,T;\VV^p(\Omega))},
\end{eqnarray*}
and then for $T \leq 1$ the result follows.
\end{proof}

In the remainder of this section it is assumed that $T\leq 1$. The following lemma shows that the mapping $\mathcal{N}$ is well-defined. As before $R(T)$ and $C_R$ are non-decreasing with respect to $T$ and $R$ respectively.

\begin{lemma} \label{welldef}
Assume that $\mathbf{A3}$ is satisfied. Let be $R>0$ and $0<T \leq 1 $. Then for all $w \in \mathcal{B}_R(T)$, the function $v$ defined by~~\eqref{defint} satisfies 
\begin{eqnarray}
\left\| (\I + \nabla v) \Sigma(v) \right\|_{\L^{p'}(0,T;\LLL^{p'}(\Omega))} & \leq & 
C\left( 1  + C_{R(T)}T^{\alpha/p}\right), \label{eststab}
\end{eqnarray}
	where $\alpha = \min(1,(p-2)/2)$.
	%, and where the constant $C_{R(T)}$ is non-decreasing with respect to $R$ and $T$.
\end{lemma}

\begin{proof}
	By H\"{o}lder's inequality, we have
	\begin{eqnarray*}
	\|  (\I+\nabla v)\Sigma(v)\|_{\LLL^{p'}(\Omega)} & \leq &
	  \|  (\I+\nabla v)\|_{\LLL^p(\Omega)} \| \Sigma(v) \|_{\LLL^{p/(p-2)}(\Omega)}, \\
	\|  (\I+\nabla v)\Sigma(v)\|_{\L^{p'}(0,T;\LLL^{p'}(\Omega))} & \leq &
	\|  (\I+\nabla v)\|_{\L^{p}(0,T;\LLL^{p}(\Omega))} 
	\| \Sigma(v) \|_{\L^{(p/2)'}(0,T;\LLL^{(p/2)'}(\Omega))}.
		%\|  (\I+\nabla v)\Sigma(v)\|_{\L^{p'}(0,T;\LLL^{p'}(\Omega))} 
		%& \leq &
		%\|  (\I+\nabla v)\|_{\L^{p}(0,T;\LLL^{p}(\Omega))} 
		%\| \Sigma(0) \|_{\L^{(p/2)'}(0,T;\LLL^{(p/2)'}(\Omega))} \\
		%& & + \|  (\I+\nabla v)\|_{\L^{p}(0,T;\LLL^{p}(\Omega))} 
		%\| \Sigma(v) - \Sigma(0) \|_{\L^{(p/2)'}(0,T;\LLL^{(p/2)'}(\Omega))}.
	\end{eqnarray*}
	From Lemma~\ref{lemmadefint}, we have $\| v\|_{\L^p(0,T;\VV^p(\Omega))} \leq R(T)$, where $R(T) = T^{1/p}(\|u_0\|_{\VV^p(\Omega)} + R)$. Hence, inequality~\eqref{ineqA3} given in assumption~{\bf A3} enables us to deduce
	\begin{eqnarray*}
	\|  (\I+\nabla v)\Sigma(v)\|_{\L^{p'}(0,T;\LLL^{p'}(\Omega))} & \leq &
	\|  (\I+\nabla v)\|_{\L^{p}(0,T;\LLL^{p}(\Omega))} 
	\left(C + C_{R(T)} \| v\|^{\alpha}_{\L^p(0,T;\VV^p(\Omega))} \right) \\
	& \leq &
	\|  (\I+\nabla v)\|_{\L^{p}(0,T;\LLL^{p}(\Omega))} 
	\left(C + C_{R(T)} T^{\alpha/p}\left(\|u_0\|_{\VV^p(\Omega)} + R\right)^{\alpha} \right),
	\end{eqnarray*}
	where $\alpha = \min(1,(p-2)/2)$. Moreover by~\eqref{estdefint} the same estimate holds for
	%, and where the constant $C_{R(T)}$ is non-decreasing with respect to $R$ and $T$. On the other hand, the same estimate holds for
	\begin{eqnarray*}
		\|  (\I+\nabla v)\|_{\L^{p}(0,T;\LLL^{p}(\Omega))}  & \leq &
		C + T^{1/p}\left(\|u_0\|_{\VV^p(\Omega)} + R\right) ,
	\end{eqnarray*}
	so that we can conclude the proof by noticing that $T \leq 1$ implies $T^{(1+\alpha)/p} \leq T^{1/p} \leq T^{\alpha/p}$.
\end{proof}

\subsection{Invariance and relative compactness of $\mathcal{N}$ in $\mathcal{B}_R(T)$}

\begin{proposition} \label{propcontraction}
	There exists $T_0 > 0$ and $R_0 >0$ such that, for all $T \leq T_0$ and $R\geq R_0$, the set $\mathcal{B}_R(T)$ is invariant under the mapping $\mathcal{N}$. Moreover, the set $\mathcal{N}(\mathcal{B}_R(T))$ is relatively compact in $\L^p(0,T;\LL(\Omega))$.
\end{proposition}

\begin{proof}
	Let us begin by proving that $\mathcal{B}_R(T)$ is invariant under $\mathcal{N}$. For $\tilde{w} \in \mathcal{B}_R(T)$, if $\tilde{u}$ denotes the function defined by~\eqref{defint}, and if $w = \mathcal{N}(\tilde{w})$, then estimate~\eqref{estcoro} of Corollary~\ref{coroLaplace} applied to system~\eqref{sysptfix} gives the inequality
	\begin{eqnarray*} 
	\left\| \dot{w} \right\|^{p'}_{\L^{p'}(0,T;\VV^{p}(\Omega)')} + 
	\| w\|^p_{\L^p(0,T;\VV^p(\Omega))} 
	& \leq & 
	C \left(\|u_1\|^2_{\LL^2(\Omega)}   + \| f \|^{p'}_{\L^{p'}(0,T;\VV^{p}(\Omega)')}
	+ \| g \|^{p'}_{\L^{p'}(0,T;\VV^{p}(\Gamma_N)')} \right. \\
	& &  \left. 
	 +  \| (\I+\nabla \tilde{u})\Sigma(\tilde{u}) \|^{p'}_{\L^{p'}(0,T; \LLL^{p'}(\Omega))} \right). 
	\end{eqnarray*}
	Furthermore, from the estimate~\eqref{eststab} we have by the Young's inequality
	\begin{eqnarray} 
		\left\| \dot{w} \right\|^{p'}_{\L^{p'}(0,T;\VV^{p}(\Omega)')} +
		\| w\|^p_{\L^p(0,T;\VV^p(\Omega))} 
		& \leq & 
		C \left(\|u_1\|^2_{\LL^2(\Omega)}   + \| f \|^{p'}_{\L^{p'}(0,T;\VV^{p}(\Omega)')}
		+ \| g \|^{p'}_{\L^{p'}(0,T;\VV^{p}(\Gamma_N)')} \right. \nonumber \\
		& & \left.
		+  1+ C_{R(T)}T^{\alpha/(p-1)} \right). \label{estCoro4}
	\end{eqnarray}
	%where $C_{R(T)}$ is non-decreasing with respect to $R$ and $T$. 
	By choosing any $R$ large enough, for instance
	\begin{eqnarray*}
	R & = & C \left(\|u_1\|_{\LL^2(\Omega)}   + \| f \|_{\L^{p'}(0,T;\VV^{p}(\Omega)')}
	+ \| g \|_{\L^{p'}(0,T;\VV^{p}(\Gamma_N)')} 
	+  3/2 \right),
	\end{eqnarray*}
	and then $T$ small enough in order to have $C_{R(T)}T^{\alpha/(p-1)} \leq 1/2$, we see that the function $w$ lies in the set $\mathcal{B}_R(T)$, and thus $\mathcal{B}_R(T)$ is invariant under $\mathcal{N}$. Moreover, estimate~\eqref{estCoro4} shows that if a $w$ lies in $\mathcal{N}(\mathcal{B}_R(T))$, then $w$ is bounded in $\L^p(0,T;\VV^p(\Omega)$, and $\dot{w}$ is bounded in $\L^{p'}(0,T;\VV^{p}(\Omega)')$. Since the embedding $\VV^p(\Omega) \hookrightarrow \LL^p(\Omega)$ is compact, the hypotheses of Corollary~4 of~\cite{Simon} (page 85) are satisfied, and then $\mathcal{N}(\mathcal{B}_R(T))$ is relatively compact in $\L^p(0,T;\LL^p(\Omega))$.
\end{proof}

It is clear that the set $\mathcal{B}_R(T)$ is a closed convex subset of $\L^p(0,T;\LL^p(\Omega))$. The consequence of Proposition~\ref{propcontraction} is the existence of a fixed point of $\mathcal{N}$, by the Schauder's theorem, which leads to the existence of a local-in-time solution for system~\eqref{mainsysreg}. We sum up this result as follows.

\begin{theorem} \label{thlocexist}
%Let $\kappa >0$ and $p>2$. 
Assume that assumption~$\mathbf{A3}$ is satisfied by the strain energy, and that $u_0 \in \VV^p(\Omega)$, $u_1 \in \LL^2(\Omega)$, $f \in \L^{p'}(0,T_0;\VV^p(\Omega)')$ and $g \in \L^{p'}(0,T_0;\VV^p(\Gamma)')$ for some $T_0 >0$. Then, if $T_0$ is small enough, system~\eqref{mainsysreg} admits a weak solution $w$, in the sense of Definition~\ref{def3} with $T_0$ in place of $T$. It defines $u(\cdot,t) = u_0 + \int_0^t w(\cdot,s)\d s$ which satisfies
\begin{eqnarray*}
u  \in  \W^{1,p}(0,T_0;\VV^p(\Omega)), & \quad &
%\dot{u}  \in  \L^p(0,T_0;\VV^p(\Omega)), \quad
\ddot{u}  \in  \L^{p'}(0,T_0;\VV^p(\Omega)').
\end{eqnarray*}
\end{theorem}

\section{Maximal time of existence of a weak solution} \label{sec4}

Throughout this section, we assume that the assumptions~$\mathbf{A1}$--$\mathbf{A4}$ are satisfied for the strain energy and the data. Then, in particular, the hypotheses of Theorem~\ref{thlocexist} hold.

\begin{lemma} \label{lemmaunifdet}
	For $0 < T < +\infty$, if $v \in \W^{1,p}(0,T;\VV^p(\Omega))$ with $p\geq d$, then the function $\chi : t \mapsto \det(\I+\nabla v(\cdot,t))$ admits a (uniformly) continuous representative function on $[0,T]$, with values in $\L^1(\Omega)$. More precisely, for $t, \ t'\in [0,T]$, we have
	\begin{eqnarray*}
	\| \chi(t) - \chi(t') \|_{\L^1(\Omega)} & \leq & 
	C| t - t' |^{1-1/p} 
	\left(1+\| \nabla u\|^{d-1}_{\L^\infty(0,T;\LLL^p(\Omega))}\right)
	\| \nabla\dot{u}\|_{\L^p(0,T;\LLL^p(\Omega))}. \label{estunifdet}
	\end{eqnarray*}
In particular, the limit $\displaystyle \lim_{t\mapsto T} \det(\I+\nabla v(\cdot,t))$ exists in $\L^1(\Omega)$.	
\end{lemma}

\begin{proof}
Since $p \geq d$, we have $\chi \in \L^{\infty}(0,T;\L^{p/d}(\Omega))$. Furthermore, if $\cof(A)$ denotes the cofactor matrix of a matrix $A$, we have
\begin{eqnarray*}
		\dot{\chi} & = & \cof(\I + \nabla v):\nabla \dot{v}, \\
		\| \dot{\chi} \|_{\L^p(0,T;\L^{p/d}(\Omega))} & \leq & 
		\| \cof(\I + \nabla v) \|_{\L^\infty(0,T;\LLL^{p/(d-1)}(\Omega))} 
		\| \nabla\dot{v}\|_{\L^p(0,T;\LLL^p(\Omega))} \\
		& \leq & C\left(1+\| \nabla v\|\right)_{\L^\infty(0,T;\LLL^p(\Omega))}^{d-1}
		\| \nabla\dot{v}\|_{\L^p(0,T;\LLL^p(\Omega))}.
\end{eqnarray*}
For $t$, $t' \in [0,T)$, since we have
\begin{eqnarray*}
	\| \chi(t) - \chi(t') \|_{\L^1(\Omega)} & \leq & | t - t' |^{1-1/p} \| \dot{\chi} \|_{\L^p(0,T;\L^1(\Omega))},
\end{eqnarray*}
the result follows.
\end{proof}

\begin{proposition} \label{propenergy}
Let $T_0>0$ be the time of existence provided by Theorem~\ref{thlocexist}, of a local-in-time solution $\dot{u}$ for system~\eqref{mainsysreg} on $(0,T_0)$, in the sense of Definition~\ref{def3}. Assume that there exists $\eta \in \L^1(\Omega)$, such that for all $t \in [0,T_0]$
\begin{eqnarray}
\det(\I + \nabla u(\cdot ,t)) \geq \eta >0, & & \text{almost everywhere in } \Omega. \label{ineqinv}
\end{eqnarray}
Then, for all $t\in [0,T_0]$, the following energy estimate holds:
\begin{eqnarray}
& & 	\| \dot{u}(t) \|^2_{\LL^2(\Omega)} 
+ \| u(t) \|^p_{\VV^p(\Omega)}
+ \kappa \int_0^{t} \|  \dot{u}(s)\|^p_{\VV^{p}(\Omega)}\d s \nonumber \\
\leq & & 	 C_0\exp(C_0t) \left( \int_{\Omega} \mathcal{W}(E(u_0))\,\d\Omega + \| u_1 \|^2_{\LL^2(\Omega)} + 
\int_0^{t} \| f(s)\|^2_{\LL^2(\Omega)} \d s + 
\int_0^{t} \| g(s)\|^2_{\VV^{2}(\Gamma_N)'} \d s \right),
  \label{energy}
\end{eqnarray}
where in particular the constant $C_0>0$ does not depend on $\kappa$.
\end{proposition}

\begin{proof}
Taking the inner product in $\LL^2(\Omega)$ by $\dot{u}$ of the first equation of system~\eqref{mainsysreg} leads to, after integration by parts
\begin{eqnarray}
\frac{\rho}{2} \frac{\d }{\d t} \left(\|\dot{u} \|^2_{\LL^2(\Omega)}\right) + 
\kappa \| \nabla \dot{u} \|^p_{\LLL^p(\Omega)} + 
\int_{\Omega} (\I+ \nabla u) \Sigma(u) : \nabla \dot{u}\, \d \Omega
& = & \int_{\Omega} f\cdot \dot{u}\,\d \Omega +  \langle g  ; \dot{u} \rangle_{\VV^{2}(\Gamma_N)'; \VV^{2}(\Gamma_N)} . \quad \label{energy0}
\end{eqnarray}
Recall from assumption~$\mathbf{A2}$ that the derivative $\check{\Sigma}$ of $\mathcal{W}$ with respect to $E$ defines a symmetric tensor, and satisfies $\check{\Sigma}(E(u)) = \Sigma(u)$. Therefore
\begin{eqnarray*}
\int_{\Omega} (\I+ \nabla u) \Sigma(u) : \nabla \dot{u}\, \d \Omega & = &
\int_{\Omega}  \Sigma(u) : \frac{1}{2}\left((\I+ \nabla u)^T\nabla \dot{u} + \nabla \dot{u}^T(\I+\nabla u)\right) \d \Omega \\
& = & \int_{\Omega}  \check{\Sigma}(E(u)) : \left(E'(u). \dot{u}\right) \d \Omega 
 =  \frac{\d }{\d t}\int_{\Omega} \mathcal{W}(E(u))\, \d \Omega. 
\end{eqnarray*}
Then, integrating~\eqref{energy0} on $(0,t)$ yields
\begin{eqnarray*}
& & \frac{\rho}{2}\| \dot{u}(t) \|^2_{\LL^2(\Omega)} 
+ \int_{\Omega} \mathcal{W}(E(u(t)))\,\d\Omega
+ \kappa \int_0^{t} \|  \dot{u}(s)\|^p_{\VV^{p}(\Omega)}\d s \\ 
= & & 
\frac{\rho}{2}\| u_1 \|^2_{\LL^2(\Omega)} + 
\int_{\Omega} \mathcal{W}(E(u_0))\,\d\Omega +
\int_0^{t} \int_{\Omega} f(s)\cdot \dot{u}(s)\, \d \Omega \, \d s +
\int_0^{t}  \langle g(s)  ; \dot{u}(s) \rangle_{\VV^{2}(\Gamma_N)'; \VV^{2}(\Gamma_N)}  \d s, \\
 \leq & & 
 \frac{\rho}{2}\| u_1 \|^2_{\LL^2(\Omega)} + 
 \int_{\Omega} \mathcal{W}(E(u_0))\,\d\Omega +
 \frac{1}{2} \int_0^{t} \left(\| f(s)\|^2_{\LL^2(\Omega)} + 
\| g(s)\|^2_{\VV^{2}(\Gamma_N)'}\right)\d t + 
C\int_0^{t} \|\dot{u}(s)\|_{\LL^2(\Omega)}^2\d t .
\end{eqnarray*}
From inequality~\eqref{ineqA1} of assumption~$\mathbf{A1}$, combined with~\eqref{ineqKorn} which holds in particular because of~\eqref{ineqinv}, we get 
\begin{eqnarray*}
 CC_K \|u(t) \|^p_{\VV^p(\Omega)} 
\leq C \|E(u(t))\|^{p/2}_{\LLL^{p/2}(\Omega)} 
\leq \int_{\Omega} \mathcal{W}(E(u(t)))\,\d\Omega,
%& & \int_{\Omega} \mathcal{W}(E(u_0))\,\d\Omega \leq  C \| E(u_0) \|^{p/2}_{\LLL^{p/2}(\Omega)}, \\
\end{eqnarray*}
where $C_0 \in \R$. The constant $C_K$ does not depend on time, because of~\eqref{ineqinv}. Combined with these inequalities, the energy estimate above then becomes
\begin{eqnarray*}
	& & \| \dot{u}(t) \|^2_{\LL^2(\Omega)} 
	+ \| u(t) \|^p_{\VV^p(\Omega)}
	+ \kappa \int_0^{t} \|  \dot{u}(s)\|^p_{\VV^{p}(\Omega)}\d s \\ 
 \leq & & 
 C\left( \| u_1 \|^2_{\LL^2(\Omega)} + 
 \int_{\Omega} \mathcal{W}(E(u_0))\,\d\Omega +
  \int_0^{t} \| f(s)\|^2_{\LL^2(\Omega)} \d s + 
 \int_0^{t} \| g(s)\|^2_{\VV^{2}(\Gamma_N)'}\d t \right)+ 
 C\int_0^{t} \|\dot{u}(s)\|_{\LL^2(\Omega)}^2\d s .
\end{eqnarray*}
%From the Gr\"{o}nwall's lemma, we deduce
%\begin{eqnarray*}
%& & 	C\left( \| \dot{u}(T_0) \|^2_{\LL^2(\Omega)} 
%	+ \| u(T_0) \|^p_{\LL^p(\Omega)} \right)
%	+ \kappa \int_0^{T_0} \|  \dot{u}(t)\|^p_{\VV^{p}(\Omega)}\d t \\
% \leq & & 	 C\exp(CT_0) \left(  
%\| u_1 \|^2_{\LL^2(\Omega)} + 
%\int_{\Omega} \mathcal{W}(E(u_0))\,\d\Omega +
% \int_0^{T_0} \| f\|^2_{\LL^2(\Omega)} \d t + 
%\int_0^{T_0} \| g\|^2_{\VV^{2}(\Gamma_N)'} \d t \right),
%\end{eqnarray*}
The proof can be concluded with the Gr\"{o}nwall's lemma.
\end{proof}

A weak solution $w = \dot{u}$ for system~\eqref{mainsys} can be defined as in Definition~\ref{def3}, with $\kappa = 0$. Without ambiguity, we still call $u$ a weak solution, determined by $\dot{u}$ and $u_0$ through 
\begin{eqnarray*}
u(\cdot,t) & = & u_0 + \int_0^t \dot{u}(\cdot,s)\d s.
\end{eqnarray*}
The energy estimate of Proposition~\ref{propenergy} enables us to prove the main result of the paper.

\begin{theorem} \label{thfinal}
Let be $p>2$, $p \geq d$. Assume that the assumptions~$\mathbf{A1}$--$\mathbf{A4}$ are satisfied. For all $T>0$, there exists a constant $C>0$ such that, if
\begin{eqnarray}
\int_\Omega \mathcal{W}(E(u_0))\,\d \Omega + \|u_1\|_{\LL^2(\Omega))} + 
\| f \|_{\L^{2}(0,T;\LL^2(\Omega))} +
\| g \|_{\L^{2}(0,T;\VV^{2}(\Gamma_N)')} & \leq & C\exp(-CT), \label{smallness}
\end{eqnarray}
then system~\eqref{mainsys} admits a weak solution $\dot{u}$ in the sense of Definition~\ref{def3} such that
\begin{eqnarray*}
u \in \L^{\infty}(0,T; \VV^p(\Omega)), \quad
\dot{u} \in \L^{\infty}(0,T; \LL^2(\Omega)), \quad
\ddot{u} \in \L^{p'}(0,T; \VV^p(\Omega)').
\end{eqnarray*}
%where $T_{max}$ is defined by the alternative:
%\begin{eqnarray*}
%	\begin{array} {ll}
%		\textrm{(a)} & \textrm{Either $T_{\max} = +\infty$,} \\
%		\textrm{(b)} & \textrm{or there exists a set $\omega \subset \Omega$ of positive measure such that $\displaystyle \lim_{t \rightarrow T_{max}} \det(\I + \nabla u(x,t)) = 0$ for almost every $x \in \omega$.}
%	\end{array}
%\end{eqnarray*}
\end{theorem}

\begin{proof}
	Let $T>0$ be arbitrary. For $\kappa >0$, we denote by $\dot{u}_\kappa$ the solution provided on $[0,T_0]$ by Theorem~\ref{thlocexist}.\\
	
	\textit{Step 1.} Consider $\eta \in \L^1(\Omega)$ such that $0 < \eta < \det(\I+\nabla u_0)$ almost everywhere in $\Omega$. For instance, we choose $\eta = \frac{1}{2}\det(\I+\nabla u_0)$. We define
	\begin{eqnarray*}
		T_{max}(\kappa) & = & \sup \left\{ T_0 >0, \text{ such that $\dot{u}_\kappa$ satisfies~\eqref{mainsysreg} on $[0,T_0]$, in the sense of Definition~\ref{def3}, and}\right. \\
		& & \qquad \ \text{for all $t \in [0,T_0)$, }\left.\det(\I+ \nabla u_\kappa(\cdot,t)) \geq \eta \text{ almost everywhere in } \Omega
		\right\}.
	\end{eqnarray*}
	Since $\det(\I+\nabla u_0) >0$, Lemma~\ref{lemmaunifdet} paired with Theorem~\ref{thlocexist} shows that $T_{max}(\kappa) >0$. Assume that $T_{max}(\kappa) < T$. We will show that this leads to a contradiction, under the announced smallness assumption on the data. Estimate~\eqref{energy} of Proposition~\ref{propenergy} -- used for $t = T_{max}(\kappa)$ -- shows that the functions $u_\kappa(T_{max}(\kappa))$ and $\dot{u}_\kappa(T_{max}(\kappa))$ are in $\VV^p(\Omega)$ and $\LL^2(\Omega)$, respectively. Hence, Theorem~\ref{thlocexist} enables us to extend $u_\kappa$ on the interval $[T_{max}(\kappa) , T_{max}(\kappa)+\tau(\kappa))$ for some $\tau(\kappa) >0$. On the other hand, estimate~\eqref{energy} shows also that, for all $\varepsilon >0$, the data can be chosen small enough, namely 
	\begin{eqnarray*}
	\int_\Omega \mathcal{W}(E(u_0))\,\d \Omega + \|u_1\|_{\LL^2(\Omega))} + 
	\| f \|_{\L^{2}(0,T;\LL^2(\Omega))} +
	\| g \|_{\L^{2}(0,T;\VV^{2}(\Gamma_N)')} & \leq & \frac{\varepsilon^p}{C_0}\exp(-C_0T)
	\end{eqnarray*}
	in order to have $\| u_\kappa(T_{max}(\kappa) \|_{\VV^p(\Omega)}  \leq  \varepsilon$. Since there exists $\hat{C}>0$ independent of $v\in \VV^p(\Omega)$ such that
	\begin{eqnarray*}
		\det(\I+\nabla v  ) \geq 1 -\hat{C}\| v\|^d_{\VV^p(\Omega)} & \text{ and } &
		\det(\I+\nabla v) \leq 1 + \hat{C}\| v\|^d_{\VV^p(\Omega)},
	\end{eqnarray*}
	we can choose $\varepsilon >0$ small enough, and further decrease $\| u_0\|_{\VV^p(\Omega)}$ if necessary in order to have
	\begin{eqnarray*}
	 \det(\I+\nabla u_\kappa(T_{max}(\kappa))) - \frac{2}{3}\det(\I+\nabla u_0)  \geq 
		\frac{1}{3} - \hat{C}\left(\frac{2}{3}\| u_0\|^d_{\VV^p(\Omega)} + 
		\varepsilon^d \right) & \geq & 0,\\
	 \det(\I+\nabla u_\kappa(T_{max}(\kappa))) & \geq & \frac{2}{3}\det(\I+\nabla u_0).
	\end{eqnarray*}
	Note that smallness for $\| u_0\|_{\VV^p(\Omega)}$ is implied by $\int_\Omega \mathcal{W}(E(u_0))\, \d \Omega$ small, due to assumption~$\mathbf{A1}$ and the Korn's inequality~\eqref{ineqKorn}. Then, by continuity (see Lemma~\ref{lemmaunifdet}), we can choose $\tau(\kappa) >0$ small enough in order to have
	\begin{eqnarray*}
	\det(\I+\nabla u_\kappa(t)) & \geq & \frac{1}{2}\det(\I+\nabla u_0), \qquad
	\text{for all } t \in [T_{max}(\kappa), T_{max}(\kappa)+\tau(\kappa) ).
	\end{eqnarray*}
	This contradicts the definition of $T_{max}(\kappa)$ as an upper bound. Thus, under the hypothesis~\eqref{smallness}, one can assume that $T_{max}(\kappa) \geq T$, for all $\kappa >0$.\\
	
	\textit{Step 2.} In order to make $\kappa$ tend to zero, estimate~\eqref{energy} gives us a $\kappa$-independent bound on $u_\kappa$ in $\L^{\infty}(0,T;\VV^p(\Omega))$, and on $\dot{u}_\kappa$ in $\L^{\infty}(0,T;\LL^2(\Omega))$. We still need a bound on $\ddot{u}_\kappa$ in $\L^{p'}(0,T;\VV^p(\Omega)')$. The variational formulation of system~\eqref{mainsysreg} -- given in Definition~\ref{def3} -- shows that we have
	\begin{eqnarray*}
		\rho \|\ddot{u}_\kappa\|_{\VV^p(\Omega)'} & \leq & 
		\kappa \| |\nabla \dot{u}_{\kappa} |^{p-1} \|_{\L^{p'}(\Omega)} 
		+ \| f \|_{\VV^p(\Omega)'}
		+ \| g \|_{\VV^p(\Gamma)'} 
		+ \| (\I + \nabla u_\kappa) \Sigma(u_{\kappa}) \|_{\LLL^{p'}(\Omega)}.
		%\rho \langle \ddot{u}_\kappa ; \varphi \rangle_{\VV^p(\Omega)' ; \VV^p(\Omega)} & = &
		%\langle f ; \varphi \rangle_{\VV^p(\Omega)' ; \VV^p(\Omega)} +
		%\langle g ; \varphi \rangle_{\VV^p(\Gamma_N)' ; \VV^p(\Gamma_N)}
	\end{eqnarray*}
	As in the proof of Lemma~\ref{welldef}, we can estimate the last term of the right-hand-side with the use of the estimate~\eqref{ineqA3} of assumption~$\mathbf{A3}$, as follows
	\begin{eqnarray}
		\rho \|\ddot{u}_\kappa\|_{\VV^p(\Omega)'} & \leq & 
		\kappa \|  \dot{u}_{\kappa}  \|^{p-1}_{\VV^{p}(\Omega)} 
		+ \| f \|_{\VV^p(\Omega)'}
		+ \| g \|_{\VV^p(\Gamma)'} 
		+  \| (\I + \nabla u_\kappa) \|_{\LLL^p(\Omega)}
		\| \Sigma(u_{\kappa}) \|_{\LLL^{(p/2)'}(\Omega)}, \nonumber \\
		\rho \|\ddot{u}_\kappa\|_{\L^{p'}(0,T;\VV^p(\Omega)')} & \leq & 
		\kappa \|  \dot{u}_{\kappa}  \|^{p-1}_{\L^p(0,T;\VV^{p}(\Omega))} 
		+ \| f \|_{\L^{p'}(0,T;\VV^p(\Omega)')}
		+ \| g \|_{\L^{p'}(0,T;\VV^p(\Gamma)')} \nonumber \\ 
		& & +  \| (\I + \nabla u_\kappa) \|_{\L^p(0,T;\LLL^p(\Omega))}
		\| \Sigma(u_{\kappa}) \|_{\L^{(p/2)'}(0,T;\LLL^{(p/2)'}(\Omega))} \nonumber \\
		& \leq & 
		\kappa \|  \dot{u}_{\kappa}  \|^{p-1}_{\L^p(0,T;\VV^{p}(\Omega))} 
		+ \| f \|_{\L^{p'}(0,T;\VV^p(\Omega)')}
		+ \| g \|_{\L^{p'}(0,T;\VV^p(\Gamma)')} \nonumber \\ 
		& & + C(T)\left(1+  \|  u_\kappa \|_{\L^p(0,T;\VV^p(\Omega))}\right)
		\left(1+ \| u_{\kappa} \|^\alpha_{\L^{p}(0,T;\VV^{p}(\Omega))}\right), \label{estuddot}
		%\label{lastest}
		%\rho \langle \ddot{u}_\kappa ; \varphi \rangle_{\VV^p(\Omega)' ; \VV^p(\Omega)} & = &
		%\langle f ; \varphi \rangle_{\VV^p(\Omega)' ; \VV^p(\Omega)} +
		%\langle g ; \varphi \rangle_{\VV^p(\Gamma_N)' ; \VV^p(\Gamma_N)}
	\end{eqnarray}
	where $\alpha = \min(1,(p-2)/2)$, and where the constant $C(T)$ is non-decreasing with respect to $T$, depends only on the bound of $\|u_\kappa \|_{\L^{p}(0,T;\VV^p(\Omega))}$, which is controlled by $\|u_\kappa \|_{\L^{\infty}(0,T;\VV^p(\Omega))}$ as follows
	\begin{eqnarray*}
	\|u_\kappa \|_{\L^{p}(0,T;\VV^p(\Omega))} & \leq & 
	T^{1/p}\|u_\kappa \|_{\L^{\infty}(0,T;\VV^p(\Omega))}.
	\end{eqnarray*}
	From~\eqref{energy}, the sequence $\left\{\|u_\kappa \|_{\L^{\infty}(0,T;\VV^p(\Omega))}; \kappa >0 \right\}$ is bounded independently of $\kappa$. Since we have
	\begin{eqnarray*}
	\kappa \|  \dot{u}_{\kappa}  \|^{p-1}_{\L^p(0,T;\VV^{p}(\Omega))} & = & \kappa^{1/p}
	\left(\kappa \|  \dot{u}_{\kappa}  \|^{p}_{\L^p(0,T;\VV^{p}(\Omega))}\right)^{1-1/p},
	\end{eqnarray*}
	the sequence $\left\{\kappa \|  \dot{u}_{\kappa}  \|^{p-1}_{\L^p(0,T;\VV^{p}(\Omega))}: \, \kappa \in (0,1] \right\}$ is bounded as well. Thus, from~\eqref{estuddot}, the sequence\\ $\left\{\|\ddot{u}_\kappa\|_{\L^{p'}(0,T;\VV^p(\Omega)')}: \, \kappa \in (0,1] \right\}$ is also bounded. By the Banach-Alaoglu theorem (see for instance~\cite{Rudin}, section~3.17), up to extraction of a subsequence, when $\kappa$ goes to zero the sequence $\{u_\kappa : \, \kappa \in (0,1] \}$ converges weakly-* to some $u$ such that
	\begin{eqnarray*}
	u\in \L^{\infty}(0,T;\VV^p(\Omega)), \quad  \dot{u} \in \L^{\infty}(0,T;\LL^2(\Omega)),
	\quad \ddot{u} \in \L^{p'}(0,T; \VV^p(\Omega)').
	\end{eqnarray*}
	
	\textit{Step 3.} By passing to the limit in the variational formulation of Definition~\ref{def3}, we see that $\dot{u}$ is a weak solution of system~\eqref{mainsys}. 
	%Besides, this solution satisfies the same energy estimate~\eqref{energy} with $\kappa = 0$. Here again, by proceeding by contradiction as previously, we can show that the maximal time of existence of $u$, $\dot{u}$ and $\ddot{u}$ in $\L^{\infty}(0,T;\VV^p(\Omega))$, $\L^{\infty}(0,T;\LL^2(\Omega))$ and $\L^{p'}(0,T;\VV^p(\Omega)')$ respectively, cannot be finite.
\end{proof}

%\section{Higher regularity for smooth data} \label{sec5}
 
\section*{Acknowledgments}
The authors gratefully acknowledge support by the Austrian Science Fund (FWF) special
research grant SFB-F32 "Mathematical Optimization and Applications in Biomedical
Sciences", and the Austrian Academy of Sciences (OAW).

%\nocite*
\bibliographystyle{alpha}
\bibliography{Existence-elasticity_ref}

\begin{thebibliography}{GMM07}

\bibitem[Age00]{Agemi}
Rentaro Agemi.
\newblock Global existence of nonlinear elastic waves.
\newblock {\em Invent. Math.}, 142(2):225--250, 2000.

\bibitem[BB69]{BardosBrezis}
Clause Bardos and Ha\"{i}m Brezis.
\newblock Sur une classe de probl\`emes d'\'evolution non lin\'eaires.
\newblock {\em J. Differential Equations}, 6:345--394, 1969.

\bibitem[BB74]{BB1974}
Jean-Pierre Bourguignon and Ha\"{i}m Brezis.
\newblock Remarks on the {E}uler equation.
\newblock {\em J. Functional Analysis}, 15:341--363, 1974.

\bibitem[Cia88]{Ciarlet}
Philippe~G. Ciarlet.
\newblock {\em Mathematical elasticity. {V}ol. {I}}, volume~20 of {\em Studies
  in Mathematics and its Applications}.
\newblock North-Holland Publishing Co., Amsterdam, 1988.
\newblock Three-dimensional elasticity.

\bibitem[CM04]{Mardare2004}
Philippe.~G. Ciarlet and Cristinel. Mardare.
\newblock Continuity of a deformation in {$H^1$} as a function of its
  {C}auchy-{G}reen tensor in {$L^1$}.
\newblock {\em J. Nonlinear Sci.}, 14(5):415--427 (2005), 2004.

\bibitem[CM12]{Cherrier}
Pascal Cherrier and Albert Milani.
\newblock {\em Linear and quasi-linear evolution equations in {H}ilbert
  spaces}, volume 135 of {\em Graduate Studies in Mathematics}.
\newblock American Mathematical Society, Providence, RI, 2012.

\bibitem[CM15]{Ciarlet2015}
Philippe~G. Ciarlet and Cristinel Mardare.
\newblock Nonlinear {K}orn inequalities.
\newblock {\em J. Math. Pures Appl. (9)}, 104(6):1119--1134, 2015.

\bibitem[Ebi93]{Ebin1993}
David~G. Ebin.
\newblock Global solutions of the equations of elastodynamics of incompressible
  neo-{H}ookean materials.
\newblock {\em Proc. Nat. Acad. Sci. U.S.A.}, 90(9):3802--3805, 1993.

\bibitem[Ebi96]{Ebin1996}
David~G. Ebin.
\newblock Global solutions of the equations of elastodynamics for
  incompressible materials.
\newblock {\em Electron. Res. Announc. Amer. Math. Soc.}, 2(1):50--59
  (electronic), 1996.

\bibitem[EG04]{Ern}
Alexandre Ern and Jean-Luc Guermond.
\newblock {\em Theory and practice of finite elements}, volume 159 of {\em
  Applied Mathematical Sciences}.
\newblock Springer-Verlag, New York, 2004.

\bibitem[FFP79]{Fung}
Y.~C. Fung, K.~Fronek, and P.~Patitucci.
\newblock Pseudoelasticity of arteries and the choice of its mathematical
  expression.
\newblock {\em American Journal of Physiology}, 237:H620--H631, 1979.

\bibitem[GK08]{Gawinecki2008}
Jerzy.~A. Gawinecki and Piotr. Kacprzyk.
\newblock Blow-up of the solution to the initial-value problem in nonlinear
  three-dimensional hyperelasticity.
\newblock {\em Appl. Math. (Warsaw)}, 35(2):193--208, 2008.

\bibitem[GMM02]{Grandmont2002}
C\'eline. Grandmont, Yvon. Maday, and Paul. M{\'e}tier.
\newblock Existence of a solution for an unsteady elasticity problem in large
  displacement and small perturbation.
\newblock {\em C. R. Math. Acad. Sci. Paris}, 334(6):521--526, 2002.

\bibitem[GMM07]{grandmont2007}
C\'eline. Grandmont, Yvon. Maday, and Paul. M{\'e}tier.
\newblock Modeling and analysis of an elastic problem with large displacements
  and small strains.
\newblock {\em J. Elasticity}, 87(1):29--72, 2007.

\bibitem[HKM76]{HuKaMa1977}
Thomas J.~R. Hughes, Tosio Kato, and Jerrold~E. Marsden.
\newblock Well-posed quasi-linear second-order hyperbolic systems with
  applications to nonlinear elastodynamics and general relativity.
\newblock {\em Arch. Rational Mech. Anal.}, 63(3):273--294 (1977), 1976.

\bibitem[Joh84]{John1984}
Fritz John.
\newblock Formation of singularities in elastic waves.
\newblock In {\em Trends and applications of pure mathematics to mechanics
  ({P}alaiseau, 1983)}, volume 195 of {\em Lecture Notes in Phys.}, pages
  194--210. Springer, Berlin, 1984.

\bibitem[Joh88]{John1988}
Fritz John.
\newblock Almost global existence of elastic waves of finite amplitude arising
  from small initial disturbances.
\newblock {\em Comm. Pure Appl. Math.}, 41(5):615--666, 1988.

\bibitem[JT08]{Xin2008}
Xin Jie and Qin Tiehu.
\newblock Almost global existence for the initial value problem of nonlinear
  elastodynamic system.
\newblock {\em J. Math. Anal. Appl.}, 339(1):517--529, 2008.

\bibitem[KP79]{Knops1979}
Robin.~J. Knops and Lawrence.~E. Payne.
\newblock Nonexistence of global solutions in nonlinear {C}auchy
  elastodynamics.
\newblock {\em Arch. Rational Mech. Anal.}, 70(2):125--133, 1979.

\bibitem[KS96]{Klainerman}
Sergiu Klainerman and Thomas~C. Sideris.
\newblock On almost global existence for nonrelativistic wave equations in
  {$3$}{D}.
\newblock {\em Comm. Pure Appl. Math.}, 49(3):307--321, 1996.

\bibitem[Lei16]{Lei2016}
Zhen Lei.
\newblock Global well-posedness of incompressible elastodynamics in two
  dimensions.
\newblock {\em Communications on Pure and Applied Mathematics},
  69(11):2072--2106, 2016.

\bibitem[LSZ15]{Lei2015}
Zhen Lei, Thomas~C. Sideris, and Yi~Zhou.
\newblock Almost global existence for 2-{D} incompressible isotropic
  elastodynamics.
\newblock {\em Trans. Amer. Math. Soc.}, 367(11):8175--8197, 2015.

\bibitem[LW15]{LeiWang}
Zhen Lei and Fan Wang.
\newblock Uniform bound of the highest energy for the three dimensional
  incompressible elastodynamics.
\newblock {\em Arch. Ration. Mech. Anal.}, 216(2):593--622, 2015.

\bibitem[MH94]{Marsden}
Jerrold~E. Marsden and Thomas J.~R. Hughes.
\newblock {\em Mathematical foundations of elasticity}.
\newblock Dover Publications, Inc., New York, 1994.
\newblock Corrected reprint of the 1983 original.

\bibitem[Mus16]{Musesti2016}
Alessandro Musesti.
\newblock A nonlinear korn inequality based on the green-saint venant strain
  tensor.
\newblock {\em Journal of Elasticity}, pages 1--6, 2016.

\bibitem[Rud91]{Rudin}
Walter Rudin.
\newblock {\em Functional analysis}.
\newblock International Series in Pure and Applied Mathematics. McGraw-Hill,
  Inc., New York, second edition, 1991.

\bibitem[Sid96]{Sideris1996}
Thomas~C. Sideris.
\newblock The null condition and global existence of nonlinear elastic waves.
\newblock {\em Invent. Math.}, 123(2):323--342, 1996.

\bibitem[Sid00a]{Sideris2000}
Thomas~C. Sideris.
\newblock Nonresonance and global existence of prestressed nonlinear elastic
  waves.
\newblock {\em Ann. of Math. (2)}, 151(2):849--874, 2000.

\bibitem[Sid00b]{Sideris1996bis}
Thomas~C. Sideris.
\newblock The null condition and global existence of nonlinear elastic waves.
\newblock In {\em Differential equations and mathematical physics
  ({B}irmingham, {AL}, 1999)}, volume~16 of {\em AMS/IP Stud. Adv. Math.},
  pages 339--345. Amer. Math. Soc., Providence, RI, 2000.

\bibitem[Sim87]{Simon}
Jacques Simon.
\newblock Compact sets in the space {$L^p(0,T;B)$}.
\newblock {\em Ann. Mat. Pura Appl. (4)}, 146:65--96, 1987.

\bibitem[ST88]{Tougeron1988}
Monique Sabl{\'e}-Tougeron.
\newblock Existence pour un probl\`eme de l'\'elastodynamique {N}eumann non
  lin\'eaire en dimension {$2$}.
\newblock {\em Arch. Rational Mech. Anal.}, 101(3):261--292, 1988.

\bibitem[ST05]{Becca2005}
Thomas~C. Sideris and Becca Thomases.
\newblock Global existence for three-dimensional incompressible isotropic
  elastodynamics via the incompressible limit.
\newblock {\em Comm. Pure Appl. Math.}, 58(6):750--788, 2005.

\bibitem[ST07]{Becca2007}
Thomas~C. Sideris and Becca Thomases.
\newblock Global existence for three-dimensional incompressible isotropic
  elastodynamics.
\newblock {\em Comm. Pure Appl. Math.}, 60(12):1707--1730, 2007.

\bibitem[Tho03]{Becca-thesis2003}
Becca Thomases.
\newblock {\em Global existence for three-dimensional nonlinear incompressible
  elastodynamics as a limit of slightly compressible materials}.
\newblock ProQuest LLC, Ann Arbor, MI, 2003.
\newblock Thesis (Ph.D.)--University of California, Santa Barbara.

\bibitem[Yin16]{Yin2016}
Silu Yin.
\newblock Global existence for a model of inhomogeneous incompressible
  elastodynamics in 2{D}.
\newblock {\em J. Differential Equations}, 260(10):7662--7682, 2016.

\bibitem[ZY09]{Zhang2009}
Zhi-Fei Zhang and Peng-Fei Yao.
\newblock Global smooth solutions and stabilization of nonlinear elastodynamic
  systems with locally distributed dissipation.
\newblock {\em Systems Control Lett.}, 58(7):491--498, 2009.

\end{thebibliography}

\end{document}